\documentclass[reqno,a4paper,12pt]{amsart} 

\usepackage{amsmath,amscd,amsfonts,amssymb}
\usepackage{mathrsfs,dsfont}
\numberwithin{equation}{section}

\def\R{\mathbb{R}}

\def\Z{\mathbb{Z}}

\def\Lam{\Lambda}
\def\1{\mathds{1}}

\renewcommand\le{\leqslant}
\renewcommand\ge{\geqslant}
\renewcommand\leq{\leqslant}
\renewcommand\geq{\geqslant}

\renewcommand\hat{\widehat}
\renewcommand\Re{\operatorname{Re}}
\renewcommand\Im{\operatorname{Im}}

\newcommand{\ft}[1]{\widehat #1}
\newcommand{\dotprod}[2]{\langle #1 , #2 \rangle}

\newcommand{\dist}{\operatorname{dist}}
\renewcommand{\div}{\operatorname{div}}

\newcommand{\define}[1]{\emph{#1}}
\newcommand{\interior}{\operatorname{int}}

\theoremstyle{plain}
\newtheorem{thm}{Theorem}[section]
\newtheorem{lem}[thm]{Lemma}
\newtheorem{corollary}[thm]{Corollary}

\newtheorem{problem}[thm]{Problem}
\newtheorem*{claim*} {Claim}

\newcommand{\thmref}[1]{Theorem~\ref{#1}}
\newcommand{\secref}[1]{Section~\ref{#1}}

\newcommand{\lemref}[1]{Lemma~\ref{#1}}

\newcommand{\corref}[1]{Corollary~\ref{#1}}
\newcommand{\exampref}[1]{Example~\ref{#1}}

\theoremstyle{definition}
\newtheorem{definition}[thm]{Definition}
\newtheorem*{definition*}{Definition}
\newtheorem*{remarks*}{Remarks}
\newtheorem*{remark*}{Remark}
\newtheorem{example}[thm]{Example}

\newenvironment{enumerate-math}
{\begin{enumerate}
\addtolength{\itemsep}{5pt}
}
{\end{enumerate}}

\newenvironment{enumerate-text}
{\begin{enumerate}
\addtolength{\itemsep}{5pt}
}
{\end{enumerate}}

\begin{document}

 \title{Fuglede's spectral set conjecture for convex polytopes}

\author{Rachel Greenfeld}
\address{Department of Mathematics, Bar-Ilan University, Ramat-Gan 52900, Israel}
\email{rachelgrinf@gmail.com}
\author{Nir Lev}
\address{Department of Mathematics, Bar-Ilan University, Ramat-Gan 52900, Israel}
\email{levnir@math.biu.ac.il}

\thanks{Research supported by ISF grant No.\ 225/13 and ERC Starting Grant No.\ 713927.}
\date{May 29, 2017}


\begin{abstract}
	Let $\Omega$ be a convex polytope in $\R^d$. We say that $\Omega$ is spectral if the
	space $L^2(\Omega)$ admits an orthogonal basis consisting of exponential
	functions. There is a conjecture, which goes back to Fuglede (1974), that $\Omega$ is
	spectral if and only if it can tile the space by translations. It is known that if $\Omega$
	tiles then it is spectral, but the converse was proved only in
	dimension $d=2$, by Iosevich, Katz and Tao. 
\par
	By a result due to Kolountzakis, if a convex polytope
	$\Omega\subset \R^d$ is spectral, then it must be centrally symmetric. We
	prove that also all the facets of $\Omega$ are centrally symmetric. These
	conditions are necessary for $\Omega$ to tile by translations.
\par
	We also develop an approach which allows us to prove that in
	dimension $d=3$, any spectral convex polytope $\Omega$ indeed tiles by
	translations. Thus we obtain that Fuglede's conjecture is true for 
	convex polytopes in $\R^3$. 
\end{abstract}

\maketitle


\section{Introduction} \label{secI1}
\subsection{}
Let $\Omega\subset \R^d$ be a bounded, measurable set of positive Lebesgue measure. A
countable set $\Lambda\subset \R^d$ is called a \define{spectrum} for $\Omega$ if the
system of exponential functions
\begin{equation}
	\label{eqI1.1}
	E(\Lambda)=\{e_\lambda\}_{\lambda\in \Lambda}, \quad e_\lambda(x)=e^{2\pi
	i\dotprod{\lambda}{x}},
\end{equation}
constitutes an orthogonal basis in $L^2(\Omega)$,
that is, the system is orthogonal and complete in the space. A set $\Omega$ which admits a spectrum
$\Lambda$ is called a \define{spectral set}. 
\par
The classical example of such a situation is
when $\Omega$ is the unit cube in $\R^d$, and $\Lambda$ is the integer lattice $\Z^d$.
Which other sets $\Omega$ are spectral? The study of this problem was initiated by Fuglede in
1974 \cite{Fug74}. For example, in that paper it was shown that a triangle and a disk in
the plane are not spectral sets. 
\par
The set $\Omega$ is said to \define{tile} the space by
translations along a countable set $\Lambda\subset\R^d$ if the family of sets
$\Omega+\lambda$ $(\lambda\in\Lambda)$  constitutes a partition of $\R^d$ up to measure
zero. In this case we will say that $\Omega+\Lambda$ is a \define{tiling}. Fuglede
observed in \cite{Fug74} the following connection between the concepts of spectrality and tiling:
\par
\define{Let $\Lambda$ be a lattice. If $\Omega+\Lambda$ is a tiling, then the dual lattice
$\Lambda^*$ is a spectrum for $\Omega$, and also the converse is true}.
\par
 Here, by a lattice we mean the image of $\Z^d$ under some invertible linear transformation,
and the dual lattice is the set of all vectors $\lambda^*$ such that
$\dotprod{\lambda}{\lambda^*}\in \Z$, $\lambda\in\Lambda$. 
\par
Fuglede conjectured that the
spectral sets could be characterized in geometric terms using the concept of tiling, in the
following way: \define{the set $\Omega$ is spectral if and only if it can tile the space by
translations}. This conjecture inspired extensive research over the years, and a number of interesting
results supporting the conjecture had been obtained. See, for example, the survey
given in \cite[Section 3]{Kol04}.
\par
On the other hand, it turned out that there also exist counter-examples to Fuglede's
conjecture. In \cite{Tao04}, Tao constructed in dimensions 5 and higher an example of a
set $\Omega$ which is spectral, but cannot tile by translations. Subsequently,
 also examples of non-spectral sets which can tile by translations were
found, and eventually the dimension in these examples was reduced up to $d\ge 3$
(see \cite[Section 4]{KM10} and the references mentioned there).  In all these examples
the set $\Omega$ is the union of a finite number of unit cubes centered at points of the
integer lattice $\Z^d$.
\subsection{}\label{secI1.2}
It is nevertheless believed that Fuglede's conjecture should  be true if the set $\Omega$ is assumed to
be \define{convex}. There is a well-known characterization due to Venkov \cite{Ven54},
that was rediscovered by McMullen \cite{McM80, McM81}, of the convex bodies (compact
convex sets with non-empty interior) which can tile the space by translations: 
\par
\begin{em}
Let $\Omega$ be a convex body in $\R^d$. Then $\Omega$ tiles by translations if and only
if the following four conditions are satisfied:
\begin{enumerate-math}
\item \label{vm:i} $\Omega$ is a polytope; 
\item \label{vm:ii} $\Omega$ is centrally symmetric; 
\item \label{vm:iii} all the facets of $\Omega$ are centrally symmetric; 
\item \label{vm:iv} each ``belt'' of $\Omega$ consists of exactly $4$ or $6$ facets. 
\end{enumerate-math}
\end{em}
\par
By a \define{belt} of a convex polytope $\Omega \subset \R^d$ with centrally 
symmetric facets one means the collection of its facets which contain a translate of a
given subfacet (that is, a $(d-2)$-dimensional face) of $\Omega$. 
\par
It was also proved in
\cite{Ven54, McM80} that if a convex polytope $\Omega$ can tile by translations, then it
admits a face-to-face tiling by translates along a certain \define{lattice}. Hence,
combined with Fuglede's theorem above this yields the following result:
\par
\define{Let $\Omega\subset\R^d$ be a convex body. If $\Omega$ tiles by translations, then
$\Omega$ is spectral}. 
\par
The converse to this result, however, is known only in dimension
$d=2$. It is due to Iosevich, Katz and Tao \cite{IKT03}, who showed that a spectral convex
body in $\R^2$ must be either a parallelogram or a centrally symmetric hexagon, and hence
it tiles by translations.
\par
 The situation in dimensions $d\ge 3$ is much less understood. It is known
that the ball is not a spectral set \cite{IKP99, Fug01}, as well as any convex body with a
smooth boundary \cite{IKT01}. We established in \cite{GriLev16a} that if $\Omega$ is a cylindric
convex body whose base has a smooth boundary, then it can neither be spectral.
\par
Kolountzakis \cite{Kol00} proved the following result:
\par
\define{Let $\Omega$ be a convex body in $\R^d$.
If $\Omega$ is spectral, then it must be centrally symmetric.}

\subsection{} 
In this paper we will focus on the case when $\Omega$ is a
\define{convex polytope}. Our first
result shows that in this case, not only the central symmetry of $\Omega$, but also the
central symmetry of all the facets of $\Omega$, is a necessary condition for spectrality:
\begin{thm}
	\label{thmI1.1}
	Let $\Omega$ be a convex polytope in $\R^d$. If $\Omega$ is a spectral
	set, then all the facets of $\Omega$ must be centrally symmetric.
\end{thm}
Our proof of this result is inspired by the paper \cite{KP02}.
\par
 Together with the result
from \cite{Kol00} we thus obtain that a spectral convex polytope $\Omega\subset\R^d$ must
satisfy the conditions \ref{vm:ii} and \ref{vm:iii} in the Venkov-McMullen
theorem above. So this supports the conjecture that any such $\Omega$ can tile by translations.
\par
Our next theorem, which is the main result of this paper, confirms that this is indeed the
case in dimension $d=3$:
\begin{thm}
	\label{thmI1.2}
	Let $\Omega$ be a convex polytope in $\R^3$. If $\Omega$ is a spectral set, then it can
	tile by translations. 
\end{thm}
Combined with the above mentioned results we thus obtain that Fuglede's conjecture is true
for convex polytopes $\Omega\subset \R^3$.
\subsection{}
In two dimensions, the convex polygons which can tile by translations are
precisely the parallelograms and the centrally symmetric hexagons. 
The three-dimensional
convex polytopes which can tile by translations were classified in 1885 by Fedorov
\cite{Fed85} into five distinct combinatorial types: the parallelepiped, the hexagonal
prism, the rhombic dodecahedron, the elongated dodecahedron and the truncated octahedron
(see, for example, \cite[Figure 32.4]{Gru07} for a graphical illustration of these types). Thus, for
a convex polytope $\Omega\subset\R^3$ to tile by translations it is necessary and
sufficient that it belongs to one of these five types, and that $\Omega$, as well as all
its facets, are centrally symmetric. A detailed exposition of this result can be found in 
\cite[Section 8.1]{Ale05}.
\par
 \thmref{thmI1.2} therefore yields that these conditions are also
necessary and sufficient for a convex polytope $\Omega\subset\R^3$ to be spectral.
\par
(The requirement that $\Omega$ is centrally symmetric is in fact redundant in
this characterization: it is known \cite{Ale33} that if all the facets of a
convex polytope $\Omega\subset \R^d$, $d\ge 3$, are centrally symmetric, then $\Omega$
itself must also be centrally symmetric.)
\subsection{}
As mentioned above, the Venkov-McMullen and Fuglede results imply not only that a convex
polytope $\Omega \subset \R^d$ which can tile by translations is necessarily spectral, 
but also that $\Omega$ admits a \define{lattice} spectrum. Our approach
allows us to establish that for certain convex polytopes,
 this spectrum is the \define{unique} one, up to translation.
\par
 First we have the following result in two dimensions: 
\begin{thm}
	\label{thmI1.3}
	Let $\Omega$ be a centrally symmetric hexagon in $\R^2$. Then $\Omega$ has a
	unique spectrum up to translation. 
\end{thm}
This result is essentially contained in \cite{IKT03}, although it was not stated
explicitly in that paper. 
\par
 The three-dimensional version of the result is the following: 
\begin{thm}
	\label{thmI1.4}
	Let $\Omega$ be a convex polytope in $\R^3$ which is spectral (and hence it
	can tile by translations), but
	which is neither a parallelepiped nor a hexagonal prism. Then $\Omega$ has a unique
	spectrum up to translation.
\end{thm}
Remark that it is necessary in these results to exclude the 
parallelograms in $\R^2$, and the parallelepipeds and the centrally symmetric  hexagonal prisms in
$\R^3$. Indeed, these convex polytopes admit infinitely many non translation-equivalent
spectra (see \cite[Section 2]{JorPed99}).

\subsection{}
The paper is organized as follows. 
\par
In \secref{secP1} we present some preliminary background.
\par
In \secref{secA1} we give a proof of the fact that a spectral convex polytope
$\Omega\subset\R^d$ must be centrally symmetric. The proof given is based on the argument
from \cite{KP02}.
\par
 In \secref{secA2} we prove that also all the facets of such an $\Omega$ are
centrally symmetric (\thmref{thmI1.1}). 
\par
In Sections \ref{secC1}--\ref{secE1} we develop an approach to show that a spectral convex
polytope $\Omega \subset \R^d$ can tile by translations. In \secref{secE2} we give a
proof, based on this approach, of the result that a spectral convex polygon $\Omega\subset\R^2$
can tile by translations. 
\par
The proof of the three-dimensional \thmref{thmI1.2} is given through Sections \ref{secF1}--\ref{secG2}.
\par
In \secref{secG3} the results concerning the uniqueness of the spectrum up
to translation are deduced (Theorems \ref{thmI1.3} and \ref{thmI1.4}).
\par
 In the last
\secref{secJ1} we give additional remarks and discuss some open problems. 


\section{Preliminaries} \label{secP1}

\subsection{Notation}
We fix some notation that will be used throughout the paper.
\par
We shall denote by $\vec e_1, \dots, \vec e_d$ the standard basis vectors in $\R^d$.
\par
As usual, $\dotprod{\cdot}{\cdot}$ and $|\cdot|$ are the Euclidean scalar product and norm in $\R^d$.
\par
For a set $A \subset \R^d$ and a vector $x \in \R^d$, we  use
$\dotprod{A}{x}$ to denote the set $\{\dotprod{a}{x} : a \in A\}$.
\par
We denote by $|\Omega|$ the Lebesgue measure of a measurable set $\Omega \subset \R^d$.
\par
The Fourier transform in $\R^d$ will be normalized as
$$\ft f (\xi)=\int_{\R^d} f (x) \, e^{-2\pi i\langle \xi,x\rangle} dx.$$

\subsection{Properties of spectra}
We recall some basic properties of spectra that will be used in the paper. 
\par 
Let $\Omega\subset\R^d$ be a bounded, measurable set of positive measure. A countable set
$\Lambda\subset\R^d$ is a spectrum for $\Omega$ if the system of exponential functions
$E(\Lambda)$ defined by \eqref{eqI1.1} is an orthogonal basis in the space $L^2(\Omega)$.
Since we have 
\[
		\dotprod{e_\lambda}{e_{\lambda'}}_{L^2(\Omega)} = \int_{\Omega} e^{-2\pi
		i\dotprod{\lambda'-\lambda}{x}}dx = \hat{\1}_\Omega(\lambda'-\lambda), 
\]
it follows that the orthogonality of $E(\Lambda)$ in $L^2(\Omega)$ is equivalent to the
condition
\begin{equation}
	\label{eqP1.2}
	\Lambda-\Lambda\subset \{\hat{\1}_\Omega=0\} \cup \{0\}.
\end{equation}
\par
A set $\Lambda\subset \R^d$ is said to be \define{uniformly discrete} if there is
$\delta>0$ such that $|\lambda'-\lambda|\ge \delta$ for any two distinct points
$\lambda,\lambda'$ in $\Lambda$. The maximal constant $\delta$ with this property is
called the \define{separation constant} of $\Lambda$, and will be denoted by
$\delta(\Lambda)$.
\par
 The condition \eqref{eqP1.2} implies that if $\Lambda$ is a spectrum
for $\Omega$ then it is a uniformly discrete set, with separation constant
$\delta(\Lambda)$ not smaller than 
\begin{equation}
	\label{eqP1.3}
	\chi(\Omega):= \min \big\{ |\xi| \;:\; \xi\in\R^d, \;\; \hat{\1}_\Omega(\xi)=0
	\big\}> 0.
\end{equation}
\par
It is easy to verify that the property of $\Lambda$ being a spectrum for $\Omega$ is
invariant under translations of both $\Omega$ and $\Lambda$. It is also easy to check that
if $\Lambda$ is a spectrum for $\Omega$, and if $A$ is an invertible $d \times d$ matrix,
then the set $(A^{-1})^\top (\Lambda)$ is a spectrum for $A(\Omega)$. 
\subsection{Limits of spectra}\label{secLimits}
Let $\Lambda_n$ be a sequence of uniformly discrete sets in $\R^d$, such that
$\delta(\Lambda_n)\ge\delta>0$. The sequence $\Lambda_n$ is said to \define{converge
weakly} to a set $\Lambda$ if for every $\varepsilon>0$ and every $R$ there is
$N$ such that 
\[
\Lambda_n\cap B_R\subset \Lambda+B_\varepsilon \quad \text{and} \quad \Lambda\cap B_R\subset
\Lambda_n+B_\varepsilon
\]
 for all $n\ge N$, where by $B_r$ we denote the open ball of
radius $r$ centered at the origin. In this case, the weak limit $\Lambda$ is also 
uniformly discrete, and moreover, $\delta(\Lambda)\ge \delta$. 
\par
By a standard diagonalization
argument one can show that given any sequence $\Lambda_n$ satisfying $\delta(\Lambda_n)\ge
\delta>0$, there is a subsequence $\Lambda_{n_j}$ which converges weakly to some
(possibly empty) set $\Lambda$.
\par
 It is known that if for each $n$ the set $\Lambda_n$ is a spectrum for $\Omega$, and if $\Lambda_n$ converges
weakly to a limit $\Lambda$, then also $\Lambda$ is a spectrum for $\Omega$.
See, for example, \cite[Section~3]{GriLev16a} where  a simple proof of this fact can be found.

\par
 The latter
fact easily implies that any spectrum $\Lambda$ of $\Omega$ must be a \emph{relatively dense} set
in $\R^d$, namely, there is $R>0$ such that every ball of radius $R$ intersects $\Lambda$.
Moreover, the constant $R=R(\Omega)$ does not depend on the spectrum $\Lambda$. Indeed, if
this was not true then there would exist a sequence $\Lambda_n$ of spectra for $\Omega$
which converges weakly to the empty set, which contradicts the fact that the weak limit must
also be a spectrum for $\Omega$. 
\subsection{Fourier expansion with respect to a spectrum}
If $\Lam$ is a spectrum for $\Omega$, then each $f\in L^2(\Omega)$ admits a Fourier
expansion with respect to the orthogonal basis $E(\Lambda)$. If we extend such a function $f$ to the whole
$\R^d$ by defining it to be zero outside of $\Omega$, then we have
$\dotprod{f}{e_{\lambda}}_{L^2(\Omega)}=\hat{f}(\lambda)$, hence the Fourier expansion of $f$
has the form 
\begin{equation}
	\label{eqF5.3}
	f=\frac{1}{|\Omega|}\sum_{\lambda\in\Lambda} \hat{f}(\lambda)e_\lambda,
\end{equation}
and the series converges in $L^2(\Omega)$. Furthermore, Parseval's equality holds, namely 
\[ \|f\|_{L^2(\Omega)}^2 = \frac{1}{|\Omega|}\sum_{\lambda\in \Lambda}|\hat{f}(\lambda)|^2. \]
\par
The following fact will be useful for us:
\begin{lem}\label{lemPL6}
For each function $f \in L^2(\Omega)$ (extended to be zero outside of $\Omega$)
 the series \eqref{eqF5.3} converges unconditionally	in $L^2$ on any bounded set to some measurable function
$\tilde{f}$ defined a.e.\ on the whole $\R^d$, and $f$ coincides with $\tilde{f}$ a.e.\ on
$\Omega$.
\end{lem}
This is a simple consequence of the following:
\begin{lem}\label{lemPL5}
	Let $\Lambda\subset \R^d$ be a uniformly discrete set, and $\{c(\lambda)\}$ be a
	sequence in $\ell^2(\Lambda)$. Then the series 
	\begin{equation}
		\label{eqPL5.2}
		\sum_{\lambda\in\Lambda} c(\lambda)e_\lambda
	\end{equation} 
	converges unconditionally in $L^2(S)$ for every bounded set $S \subset
	\R^{d}$. 
\end{lem}
The latter fact is well-known, see for instance \cite[Section 4.3, Theorem 4]{You01}
where it is proved in dimension one. For the reader's convenience we provide
a self-contained proof in arbitrary dimension $d$.
\begin{proof}[Proof of \lemref{lemPL5}]
	First we show that if $S$ is a bounded set then there is a constant
	$C=C(\Lam,S)$, such that for every sequence $\{c(\lambda)\}$ with only finitely many
	non-zero terms we have
	\begin{equation}
		\label{eqPL5.1}
		\Big\| \sum_{\lambda\in \Lambda}
		c(\lambda)e_\lambda\Big\|_{L^2(S)}^2\le C
		\sum_{\lambda\in\Lambda}|c(\lambda)|^2.
	\end{equation}
	Indeed, let $\delta >0$ denote the separation constant of
	$\Lambda$, and choose a smooth function $\varphi$ supported on a ball of radius
	$\delta/2$ around the origin, such that $\int|\varphi(t)|^2 dt=1$, and 
	\[ \eta:= \inf_{x\in S} |\hat{\varphi}(x) | > 0. \]
	Then the left-hand side of \eqref{eqPL5.1} is not greater than $1 / \eta^2$ times 
\[
	\int_{\R^d}\Big|\hat{\varphi}(x)\sum_{\lambda\in\Lambda}c(\lambda)e_\lambda(x)\Big|^2 dx =
	\int_{\R^d}\Big|\sum_{\lambda\in\Lambda}c(\lambda)\varphi(t+\lambda)\Big|^2	dt 
	= \sum_{\lambda\in\Lambda}|c(\lambda)|^2,
\]
	hence \eqref{eqPL5.1} holds with $C=1/\eta^2$. 
	\par
	Now it follows from  \eqref{eqPL5.1}  that given an arbitrary sequence $\{c(\lambda)\}$ in
	$\ell^2(\Lambda)$, the partial sums of the series \eqref{eqPL5.2} constitute a
	Cauchy sequence in $L^2(S)$ for every arrangement of the terms of the series,
	and the limit in $L^2(S)$ of these partial sums is the same for every such arrangement.
	This confirms the assertion of the lemma. 
\end{proof}

\subsection{Convex polytopes}
By a \emph{convex polytope} $\Omega$ in $\mathbb R^d$ we mean a compact
set which is the convex hull of a finite number of points.
By a \emph{facet} of $\Omega$  we refer to a $(d-1)$-dimensional face of $\Omega$,
while a \emph{subfacet} is a $(d-2)$-dimensional face. 
\par
 If $G$ is a $k$-dimensional face
of $\Omega$ $(0 \leq k \leq d)$ then $|G|$ denotes the $k$-dimensional volume of $G$.
 For a facet $F$  of $\Omega$  we denote  by $\sigma_F$ the surface measure on $F$.
\par
The interior of $\Omega$ will be denoted by $\interior(\Omega)$.
\par
We say that $\Omega$ is \emph{centrally symmetric} if there
is a point $x \in \R^d$ (the center) such that $  \Omega - x = -\Omega + x$.
The following theorem, due to Minkowski, gives a criterion for the
central symmetry of a convex polytope $\Omega$ in terms of the areas of its facets:
\begin{thm}[Minkowski]
	\label{thmP3.1}
	A convex polytope $\Omega$ is centrally symmetric if and only if for
	each facet $F$ of $\Omega$ there is a parallel facet $F'$ such that $|F|=|F'|$.
\end{thm}
This is a consequence of the classical Minkowski's uniqueness theorem,
see for example \cite[Section 18.2]{Gru07}.
\par
We shall need some well-known facts about Fourier transforms related to convex polytopes 
in  $\R^d$ (actually, in some of these results the convexity is not necessary). Since the 
proofs are not difficult, they are included for completeness.
 \begin{lem}\label{lemP1}
 Let $\Omega$ be a convex polytope in $\mathbb R^d$ $(d \geq 1)$. 
For each facet $F$ of $\Omega$, let $n_F$ denote the outward unit normal to $\Omega$ on $F.$ Then
 \begin{equation}\label{eqP1.1}
 -2\pi i\xi \, \hat{\1}_\Omega(\xi)=\sum n_F \, \hat\sigma_F(\xi),\quad \xi\in\mathbb R^d,
 \end{equation}
 where the sum is over all the facets $F$ of $\Omega.$
 \end{lem}
\begin{proof} Fix two vectors $\xi$ and $u$ in $\mathbb R^d,$ and let
$$\Phi(x):=u \, e^{-2\pi i\langle \xi,x\rangle},\quad x\in\mathbb R^d.$$
Then we have
$$\div \Phi(x)=-2\pi i\langle\xi,u\rangle e^{-2\pi i\langle \xi,x\rangle}.$$
By the divergence theorem,
$$\int_\Omega \div \Phi(x)dx=\int_{\partial \Omega}\langle \Phi(x),n(x)\rangle \,  d\sigma(x),$$
where $\sigma$ denotes the surface measure on the boundary $\partial \Omega$, and 
$n(x):=n_F$ if $x$ belongs to the relative interior of a facet $F$ of $\Omega$. This means that
$$-2\pi i\langle\xi,u\rangle\hat{\1}_\Omega(\xi)=\sum\langle n_F,u\rangle\hat\sigma_F(\xi),$$
where the sum is over all the facets $F$  of $\Omega$. But since $\xi$ and $u$ were
arbitrary vectors  in $\R^d$, this implies \eqref{eqP1.1}.
\end{proof}

\begin{corollary}\label{corP2}
If $\Omega$ is a convex polytope in $\mathbb R^d$ $(d \geq 1)$, then
$$|\hat{\1}_\Omega(\xi)|\le\frac{|\partial\Omega|}{2\pi}\cdot|\xi|^{-1},$$
where $|\partial\Omega|$ denotes the total surface area of $\Omega.$
\end{corollary}

This follows from \lemref{lemP1} using the fact that the right-hand side of
\eqref{eqP1.1} is bounded in norm by $|\partial\Omega|.$

 \begin{lem}\label{lemP3}
 Let $\Omega$ be a convex polytope in $\mathbb R^d$ $(d \geq 2)$, and   $F$ be a facet of $\Omega$. 
 Let $\theta(\xi,F)$ denote the angle between a non-zero vector $\xi\in\mathbb R^d$ and the outward normal vector to $\Omega$ on $F.$
Then
$$\left|\hat\sigma_F(\xi)\right|\le\frac{|\partial F|}{2\pi}\cdot\frac{|\xi|^{-1}}{|\sin \theta(\xi,F)|},$$
where $|\partial F|$ is the $(d-2)$-dimensional volume of the relative boundary of $F$.
 \end{lem}

\begin{proof} By applying a rotation and a translation we may assume that $F$ is contained in the hyperplane $\{x_1=0\}$, and that the outward unit normal to $\Omega$ on $F$ is $\vec e_1.$ Hence
$$\hat\sigma_F(\xi)=\varphi_F(\xi_2,\xi_3,\dots,\xi_d),$$
where $\varphi_F$ denotes the Fourier transform of the indicator function of the polytope in $\mathbb R^{d-1}$ obtained by projecting the facet $F$ on the $(x_2,x_3,\dots,x_d)$ coordinates.
Using \corref{corP2}, this implies that
$$\left|\hat\sigma_F(\xi)\right|\le\frac{|\partial F|}{2\pi}
\Big( \sum_{j=2}^d\xi_j^2 \Big)^{-1/2}.$$
However, since we have
$$\xi_1=\langle\xi, \vec e_1\rangle=|\xi| \cos\theta(\xi,F),$$
it follows that
$$\sum_{j=2}^d\xi_j^2=|\xi|^2-\xi_1^2 = |\xi|^2 (1-\cos^2\theta(\xi,F)) = |\xi|^2 \sin^2\theta(\xi,F),$$
so this proves the claim.
\end{proof}
The previous lemmas imply the following result, that will be used in the next sections:
 \begin{lem}\label{lemP4}
 Let $\Omega$ be a convex polytope in $\mathbb R^d$ $(d \geq 2)$.
Assume that $A$ and $B$ are two parallel facets of $\Omega$,
and that the outward unit normals to $\Omega$ on $A$ and $B$ are respectively the vectors $\vec e_1$ and $-\vec e_1$
(we also allow $A$ to be a facet which does not have a parallel facet, in which case
we understand $B$ to be the empty set). Then there is $\alpha = \alpha(\Omega)>0$ such that 
\begin{equation}
	 -2\pi i\xi_1\hat{\1}_\Omega(\xi)=\hat\sigma_A(\xi)-\hat\sigma_B(\xi)+O(|\xi_1|^{-1}),\quad |\xi_1|\to \infty 
	\label{eqP4.3}
\end{equation}
in the cone 
\begin{equation}
	K(\alpha) := \left\{\xi\in \R^d\; : \; |\xi_j|\le \alpha|\xi_1| \quad (2 \le j \le d)  \right\}.
	\label{eqP4.1}
\end{equation}
 \end{lem}

\begin{proof} 
By \lemref{lemP1} we have
\begin{equation}
 -2\pi i\xi_1 \, \hat{\1}_\Omega(\xi)=\hat\sigma_A(\xi)-\hat\sigma_B(\xi) + \sum \dotprod{n_F}{\vec e_1} \, \hat\sigma_F(\xi),
	\label{eqP4.2}
\end{equation}
 where the sum is over all the facets $F$ of $\Omega$ other than $A$ and $B$.
	If $\alpha$ is sufficiently small, then the angle between any vector in
	$K(\alpha)$ and the outward normal to $\Omega$ on any facet $F$ other than $A$ and
	$B$, is bounded away from $0$ and $\pi$. Hence by \lemref{lemP3}, the sum
	on the right-hand side of \eqref{eqP4.2} is $O(|\xi|^{-1})$ as $|\xi|\to\infty$ in
	the cone $K(\alpha)$. But since the ratio $|\xi_1|  / |\xi|$ is bounded from 
	below in $K(\alpha)$,  this implies \eqref{eqP4.3}.
\end{proof}


\section{Spectral convex polytopes are symmetric} \label{secA1}
\subsection{}
In this section we give a proof of the following result:
\begin{thm}[Kolountzakis \cite{Kol00}]
\label{thmA1}
Let $\Omega$ be a convex polytope in $\mathbb R^d$ $(d \ge 2)$.
If $\Omega$ is spectral then $\Omega$ is centrally symmetric.
\end{thm}
In fact, it was proved in \cite{Kol00} that any convex body (not assumed to be a polytope) 
which is spectral, must be centrally symmetric. This supports the conjecture that a spectral 
convex body $\Omega$ can tile by translations, as the central symmetry is a necessary condition for
$\Omega$ to tile, by the Venkov-McMullen theorem.
\par
There is another approach to prove \thmref{thmA1}, that was introduced in
the paper \cite{KP02} due to Kolountzakis and
Papadimitrakis. This approach is specific for polytopes, but on the other hand
it does not require $\Omega$ to be convex. 
The main result in \cite{KP02} gives a certain condition 
on a polytope $\Omega \subset \R^d$ that is necessary for its spectrality.
If the  polytope $\Omega$ is  convex, then this condition coincides with the requirement that 
$\Omega$ is  centrally symmetric.
\par
For the completeness of our exposition, below we give 
a proof of \thmref{thmA1} based on the argument in \cite{KP02}.
See also \cite[pp.\ 184--185]{Kol04}. The proof may also serve as a preparation for the next section,
where the argument will be further developed. 

\subsection{Proof of \thmref{thmA1}}
By Minkowski's \thmref{thmP3.1} it would be enough to show that
for each facet $A$ of $\Omega$ there is a parallel facet $B$ such that $|A|=|B|$.
If this is not true, then there is a facet $A$ of $\Omega$ whose parallel facet $B$
satisfies $|A|>|B|,$ where we understand $B$ to be the empty set if $A$ is a facet of
$\Omega$ with no parallel facet.
\par
By applying an affine transformation, we may assume that $A$ is contained in the hyperplane $\{x_1=0\}$, that $B$ is contained in the hyperplane $\{x_1=-1\}$, and that the outward unit normals to $\Omega$ on $A$ and $B$ are respectively the vectors $\vec e_1$ and $-\vec e_1.$
It follows that
\begin{align}
\label{eqA1.1} \hat\sigma_A(\xi)&=\varphi_A(\xi_2,\xi_3,\dots,\xi_d),\\[4pt]
\label{eqA1.2} \hat\sigma_B(\xi)&=e^{2\pi i\,\xi_1}\varphi_B(\xi_2,\xi_3,\dots,\xi_d),
\end{align}
where $\varphi_A,\varphi_B$ are respectively the Fourier transforms of the indicator
functions of the polytopes in $\mathbb R^{d-1}$ obtained by projecting the facets 
$A,B$ on the $(x_2,x_3,\dots,x_d)$ coordinates.
In particular, $\varphi_A$ and $\varphi_B$ are continuous functions, and
\begin{equation}\label{eqA1.3}
\varphi_A(0)=|A|,\quad \varphi_B(0)=|B|.
\end{equation}
\par
For any $r>0$ we denote by $S(r)$ the cylinder of radius $r$ along the $x_1$-axis, namely
$$S(r) := \{t \vec e_1+w \, : \, t\in\mathbb R, \; w \in\mathbb R^d, \; |w|<r\}.$$
Notice that
\begin{equation}\label{eqA1.4}
S(r)-S(r) = S(2r).
\end{equation}
By assumption, we have $|A|>|B|.$ Choose a number $\eta$ such that
$$0<\eta<|A|-|B|.$$
It follows from \eqref{eqA1.1}, \eqref{eqA1.2} and \eqref{eqA1.3} that there is $\varepsilon >0$  such that
$$\left|\hat\sigma_A(\xi) -\hat\sigma_B(\xi) \right|\ge\left|\hat\sigma_A(\xi)\right|-\left|\hat\sigma_B(\xi)\right|>\eta, \quad \xi \in S(2\varepsilon).$$
By \lemref{lemP4} we have
$$-2\pi i\xi_1\hat{\1}_\Omega(\xi)=\hat\sigma_A(\xi)-\hat\sigma_B(\xi)+O(|\xi_1|^{-1}),\quad |\xi_1|\to \infty$$
in the cylinder $S(2\varepsilon)$. It follows that there is $R>0$ such that
\begin{equation}\label{eqA1.5}
\hat{\1}_\Omega(\xi)\ne0,\quad \xi\in S(2\varepsilon) \setminus B_R
\end{equation}
where $B_R$ denotes the ball of radius $R$ centered at the origin.

Now let $\Lambda$ be a spectrum for $\Omega.$ We claim that for any $\tau\in\mathbb R^d,$
if $\lambda,\lambda'$ are two points in $\Lambda\cap(S(\varepsilon)+\tau)$, then
$|\lambda'-\lambda|\le R$. Indeed, if not then using \eqref{eqA1.4} we get
$$\lambda'-\lambda \in S(2\varepsilon)\setminus B_R,$$
but due to \eqref{eqA1.5} this implies that $\hat{\1}_\Omega(\lambda'-\lambda)\ne0,$ a contradiction.

Since $\Lambda$ is a uniformly discrete set, it follows  that $\Lambda\cap(S(\varepsilon)+\tau)$ is a finite set, for every $\tau\in\mathbb R^d.$
Since $\Lambda$ is a relatively dense set, there is $M>0$ such that every ball of radius
$M$ intersects $\Lambda.$ The cylinder $S(M)$ may be covered by a finite number of
cylinders $S(\varepsilon)+\tau_j$ $(1\leq j \leq N)$; hence $\Lambda\cap S(M)$ is also a
finite set. But this implies that $S(M)$ must contain a ball of radius $M$ free from points of
$\Lambda,$ a contradiction. This completes the proof of \thmref{thmA1}.
\qed


\section{Spectral convex polytopes have symmetric facets} \label{secA2}
\subsection{} \label{subsecA2.1}
The result in \secref{secA1} shows that the central symmetry is a necessary 
condition for a convex polytope $\Omega\subset\R^d$ to be spectral. In the 
present section we prove that also the central symmetry of all the facets 
of $\Omega$ is necessary for spectrality: 
\begin{thm}\label{thmA2}
Let $\Omega$ be a convex, centrally symmetric polytope in $\mathbb R^d$ $(d \ge 3)$.
If $\Omega$ is spectral then all the facets of  $\Omega$ are also  centrally symmetric.
\end{thm}

Recall that by the Venkov-McMullen theorem, the central symmetry of the facets is also
a necessary condition for $\Omega$ to tile by translations. Hence this result further
supports the conjecture that any spectral convex polytope $\Omega$
can tile by translations.
\par
Notice that the conclusion cannot be further improved by showing that also
all the $k$-dimensional faces of  $\Omega$, for some $2 \le k \le d-2$, 
are centrally symmetric. Indeed, this would imply \cite{McMul70} that
all the faces of  $\Omega$ of every dimension are centrally symmetric.
However, the $24$-cell in $\R^4$ is a well-known example of a convex 
polytope which tiles by translations, and hence it is spectral, but which
does not satisfy this property.
\par
The rest of this section is devoted to the proof of  \thmref{thmA2}. 
The proof is based on a development of the argument in \cite{KP02}.

\subsection{}
Let $F$ be one of the facets of $\Omega$. As before, to prove that $F$ is centrally symmetric
it would be enough, by Minkowski's \thmref{thmP3.1}, to show that for each subfacet $A$ of $F$ 
there is a parallel subfacet $B$ of $F$ such that $|A|=|B|$. So, again, suppose to the contrary 
that $A,B$ are two parallel subfacets of $F$ such that $|A|>|B|$,
 with the agreement that $B$ is empty if $A$ has no parallel subfacet of $F.$

By applying an affine transformation, we may assume that
\begin{equation}\label{eqA2.1}
\Omega=-\Omega,
\end{equation}
namely, $\Omega$ is symmetric about the origin,
\begin{equation}\label{eqA2.2}
F\subset\{x_1=\tfrac{1}{2}\}
\end{equation}
and the outward unit normal to $\Omega$ on $F$ is $\vec e_1,$
\begin{align}
\label{eqA2.3} A & \subset\{x_1=\tfrac{1}{2},\, x_2=0\},\\[4pt]
\label{eqA2.4} B & \subset\{x_1=\tfrac{1}{2},\, x_2=-1\},
\end{align}
and the outward unit normals to $F$ on $A$ and $B$ are respectively $\vec e_2$ and $-\vec e_2.$

\subsection{}
Let $\varphi_F$ (respectively, $\varphi_A$ and $\varphi_B$) denote the Fourier transform of the indicator function of the polytope in $\mathbb R^{d-1}$ (respectively $\mathbb R^{d-2})$ obtained by projecting the facet $F$ on the $(x_2,x_3,\dots,x_d)$ coordinates (respectively, the subfacets $A$ and $B$ on the $(x_3,\dots,x_d)$ coordinates).
Define
\begin{equation}\label{eqA2.9.1}
\psi(\xi):=\Re\left[e^{-\pi i\xi_1}(\varphi_A(\xi_3,\dots,\xi_d)-e^{2\pi i\, \xi_2}\varphi_B(\xi_3,\dots,\xi_d)) \right],
\quad \xi \in \R^d. 
\end{equation}
Also, for any three positive real numbers $L$, $\delta$ and $\alpha$, we let
$$K(L,\delta,\alpha):=\{\xi \in  \mathbb R^d \; : \; L \le|\xi_2|\le\delta|\xi_1|,  \; \; |\xi_j|\le \alpha|\xi_2|\; (3\le j\le d)\}.$$

\begin{lem}\label{lemA2.5}
There is $\alpha>0$ such that given any $\eta>0$ one can find  $\delta>0$ and $L$ such that
\begin{equation}\label{eqA2.9}
\left|2\pi^2\xi_1\xi_2\hat{\1}_\Omega(\xi)+\psi(\xi)\right|<\eta, \quad \xi\in K(L,\delta,\alpha).
\end{equation}
\end{lem}

\begin{proof}
Due to \eqref{eqA2.1}, the facet of $\Omega$ parallel to $F$ is $-F$. If $0<\delta \leq \alpha <1$, then the set $K(L,\delta,\alpha)$ is contained in the cone
\begin{equation}\label{eqA2.5.1}
\{|\xi_j|\le \alpha|\xi_1|,\; 2\le j\le d\}.
\end{equation}
Hence by \lemref{lemP4}, if $\alpha$ is sufficiently small then 
\begin{equation}\label{eqA2.6}
-2\pi i\xi_1\hat{\1}_\Omega(\xi)=\hat\sigma_F(\xi)-\hat\sigma_{-F}(\xi)+O(|\xi_1|^{-1}),\quad |\xi_1|\to \infty
\end{equation}
in the cone \eqref{eqA2.5.1}. Observe that by \eqref{eqA2.2} we have
\begin{equation}\label{eqA2.7}
\hat\sigma_F(\xi)-\hat\sigma_{-F}(\xi) =2i\Im\left[\hat \sigma_F(\xi)\right] 
 =2i\Im\left[e^{-\pi i\,\xi_1}\varphi_F(\xi_2,\xi_3,\dots,\xi_d)\right].
\end{equation}
Now if $\xi\in K(L,\delta,\alpha)$ then the vector $(\xi_2,\xi_3,\dots,\xi_d)$ belongs to the cone
\begin{equation}\label{eqA2.5.2}
\{|\xi_j|\le \alpha|\xi_2|,\, 3\le j\le d\}\subset \mathbb R^{d-1},
\end{equation}
so again by \lemref{lemP4} and by \eqref{eqA2.3}, \eqref{eqA2.4}
it follows that if $\alpha$ is sufficiently small then
\begin{equation}\label{eqA2.8}
-2\pi i\xi_2\varphi_F(\xi_2,\xi_3,\dots,\xi_d)=
\varphi_A(\xi_3,\dots,\xi_d)-e^{2\pi i\, \xi_2}\varphi_B(\xi_3,\dots,\xi_d)+O(|\xi_2|^{-1}),\quad |\xi_2|\to \infty
\end{equation}
in the cone \eqref{eqA2.5.2}.
Combining \eqref{eqA2.6}, \eqref{eqA2.7} and \eqref{eqA2.8} shows that there is $\alpha>0$ and a positive constant $C$, such that
for any $0<\delta \leq \alpha$ and any $L>0$ we have
$$\left|2\pi^2\xi_1\xi_2\hat{\1}_\Omega(\xi)+\psi(\xi)\right|\le C(|\xi_2/\xi_1|+|1/\xi_2|),  \quad \xi\in K(L,\delta,\alpha).$$
But for $\xi\in K(L,\delta,\alpha)$ we have $|\xi_2/\xi_1| \leq \delta$ and $|1/\xi_2| \leq L^{-1}$.
Hence given any $\eta>0$, by choosing $\delta$ sufficiently small and $L$ sufficiently
large, we obtain \eqref{eqA2.9}.
\end{proof}

\subsection{}
Recall that, by assumption, we have $|A|>|B|$. Choose a number $\eta$ such that
$$0<2\eta<|A|-|B|.$$
Use \lemref{lemA2.5} to find $L$, $\delta$ and $\alpha$ such that \eqref{eqA2.9} holds.
Define the vector
\begin{equation}\label{eqA2.10}
v_\delta := 2 \vec e_1 + \delta \vec e_2 = \left(2,\delta,0,0,\dots,0\right).
\end{equation}

For any $r>0$ we denote by $E(r,\delta)$ the union of balls of radius $r$ centered at the integral multiples of the vector $v_\delta$, that is,
\begin{equation}\label{eqA2.11b}
E(r,\delta):=\{k v_\delta +w \,:\, k\in\mathbb Z,\, w \in\mathbb R^d, \, |w|<r\}.
\end{equation}
Notice that
\begin{equation}\label{eqA2.11a}
E(r,\delta)-E(r,\delta) = E(2r,\delta).
\end{equation}

Since $\varphi_A,$ $\varphi_B$ are continuous functions satisfying
$$\varphi_A(0)=|A|,\quad\varphi_B(0)=|B|,$$
it follows from \eqref{eqA2.9.1}, \eqref{eqA2.10} and \eqref{eqA2.11b} that there is $\varepsilon>0$ such that
$$\left|\psi(\xi)-\Re\left[|A|-e^{2\pi i\, \xi_2}|B|\right]\right|<\eta, \quad \xi\in E(2\varepsilon,\delta).$$
In particular, this implies that
\begin{equation}\label{eqA2.12}
|\psi(\xi)|\ge|A|-|B|-\eta>\eta, \quad \xi\in E(2\varepsilon,\delta).
\end{equation}

\subsection{}

\begin{lem}\label{lemA2.11}
There is $R>0$ such that 
\begin{equation}\label{eqA2.11.1}
E(2\varepsilon,\delta) \setminus B_R \subset K(L,\delta,\alpha),
\end{equation}
where $B_R$ denotes the ball of radius $R$ centered at the origin.
\end{lem}

This can be verified easily, so we skip the proof.

\subsection{}
 Now suppose that $\Lambda$ is a spectrum for $\Omega$. Use \lemref{lemA2.11} to choose $R$
 such that \eqref{eqA2.11.1} holds.
 We claim that for any $\tau\in\mathbb R^d,$ if $\lambda,\lambda'$ are two points in
 $\Lambda\cap(E(\varepsilon,\delta)+\tau)$, then $|\lambda'-\lambda|\le R.$ Indeed, if not
 then using \eqref{eqA2.11a} we get
$$\lambda'-\lambda\in E(2\varepsilon,\delta) \setminus B_R  \subset K(L,\delta,\alpha).$$
It thus follows from \eqref{eqA2.9} and \eqref{eqA2.12} that $\hat{\1}_\Omega(\lambda'-\lambda)\ne0,$ a contradiction.

Since $\Lambda$ is a uniformly discrete set, it follows  that $\Lambda\cap(E(\varepsilon,\delta)+\tau)$ is a finite set, for every $\tau\in\mathbb R^d.$
Since $\Lambda$ is a relatively dense set, there is $M>0$ such that every ball of radius $M$ intersects $\Lambda.$
Let $S(M,\delta)$ denote the cylinder of radius $M$ along the vector $v_\delta$, 
$$S(M,\delta) := \{t v_\delta+w \,:\, t\in\mathbb R, \; w \in\mathbb R^d, \; |w|<M\}.$$
Then $S(M,\delta) $ may be covered by a finite number of sets
$E(\varepsilon,\delta)+\tau_j$ $(1\le j\le N)$; hence $\Lambda\cap S(M,\delta)$ is also a
finite set. It follows that $S(M,\delta)$ contains a ball of radius $M$ free from points
of $\Lambda$, a contradiction. This completes the proof of \thmref{thmA2}.
\qed


\section{Covering and packing} \label{secC1}
It was shown in Sections \ref{secA1} and \ref{secA2} that if a convex polytope
$\Omega\subset\R^d$ is spectral, then it must be centrally symmetric and have
centrally symmetric facets. In order to prove that $\Omega$ tiles by translations,
a conceivable strategy may therefore be to try and show that every belt of 
$\Omega$ must consist of either $4$ or $6$ facets. Indeed, this 
would imply that $\Omega$ tiles, by the Venkov-McMullen theorem.
\par
Our approach, however, will not be based on such a strategy. Instead, we will use
another condition, given in terms of the spectrum $\Lam$, which implies that
$\Omega$ tiles by translations. In this section, we prove the sufficiency of
this condition (\corref{corC3}).

\subsection{}
Let $\Omega\subset \R^d$ be a convex polytope, which is centrally symmetric and has
centrally symmetric facets. If $F$ is any facet of $\Omega$, then by the central symmetry,
the opposite facet $F'$ is a translate of $F$. We shall denote by $\tau_F$ 
the translation vector in $\R^d$ which carries $F'$ onto $F$.
\par
Following \cite{Ven54, McM80} we consider the set
\begin{equation}
	T = T(\Omega) = \left\{ \sum_{F}  k_F \, \tau_F\;:\; k_F\in \Z \right\},
	\label{eqC1.1}
\end{equation}
that is, $T$ is the set of all linear combinations with integer coefficients of the
vectors $\tau_F$, where $F$ goes through all the facets of $\Omega$. The set
$T$ is a countable subgroup of $\R^d$. 
\begin{thm}[\cite{Ven54, McM80}] 
 \label{thmC1}
	$\Omega + T$ is a covering, that is, each point in
	$\R^d$ belongs to at least one of the sets $\Omega+\tau$, $\tau\in T$. 
\end{thm}
This is a part of the Venkov-McMullen theorem, which characterizes the 
convex bodies that tile by translations by the four conditions \ref{vm:i}--\ref{vm:iv}
mentioned in \secref{secI1.2}. In the sufficiency part of the theorem it is shown 
that these four conditions imply that $\Omega + T$ is a tiling. However the last condition,
namely the requirement \ref{vm:iv} that each belt consists of exactly $4$ or $6$ facets,
is not used in that part of the proof where it is shown that $\Omega + T$ is a covering
(see \cite[pp.\ 115--116]{McM80} where the latter fact is also mentioned explicitly).
Hence the proof yields that the first three conditions \ref{vm:i}--\ref{vm:iii}
are enough to conclude that $\Omega + T$ is a covering, as stated in \thmref{thmC1}.
\par
Observe that \thmref{thmC1} implies that $T$ is a relatively dense set in $\R^d$.
\par
 It also follows from this theorem that, in order to prove that $\Omega$ tiles by translations,
it would be enough to show that $\Omega+T$ is a \define{packing}, which means that the
sets $\Omega+\tau$, $\tau\in T$, are disjoint up to measure zero. Indeed, in such a case
$\Omega+T$ is simultaneously a covering and a packing, hence $\Omega$ tiles
by translations along the set $T$. 
\par
Notice that if $\Omega+T$ is a packing (and hence a
tiling), then $T$ must be a uniformly discrete set in $\R^d$. So in this case $T$ is a
subgroup of $\R^d$ which is both uniformly discrete and relatively dense,
and it follows that $T$ must be a \define{lattice}. As mentioned in
\cite{McM80}, the tiling by translations of $\Omega$ along the lattice $T$
constitutes a \emph{face-to-face} tiling.
\par
\subsection{}
The next lemma gives a sufficient condition for
$\Omega+T$ to be a packing: 
\begin{lem} \label{lemC2}
	Suppose that $\Lambda\subset \R^d$ is a set satisfying the following condition: 
	\begin{equation}
		\dotprod{\Lambda-\Lambda}{\tau_F}\subset \Z  
		\label{eqC1.2}
	\end{equation}
	for every facet $F$ of $\Omega$. If the system of exponentials $E(\Lambda)$ is
	complete in $L^2(\Omega)$, then $\Omega+T$ is a packing. 
\end{lem}
\begin{proof}
By translating $\Lambda$ we may assume that it contains the origin, hence 
$ \dotprod{\Lambda}{\tau_F}\subset \Z $
for every facet $F$. It follows that the exponential functions $e_\lambda$ ($\lambda\in
\Lambda$) are periodic with respect to $T$, namely 
\[ e_\lambda(x+\tau) = e_\lambda(x) \]
for every $\tau\in T$. If $\Omega+T$ is not a packing then there exist distinct vectors
$\tau',\tau''\in T$ such that the set 
$ (\Omega+\tau') \cap (\Omega+\tau'') $
has positive measure. Thus the set $E$ defined by 
\[ E:= \Omega\cap (\Omega-\tau),  \quad \tau:=\tau''-\tau' \]
is a set of positive measure, and $E$, $E+\tau$ are both contained in $\Omega$.
Hence the function $ f:= \1_E-\1_{E+\tau} $ is supported by $\Omega$,
and since $\tau\ne 0$, the function $f$ does not vanish identically a.e.
On the other hand, for every
$\lambda\in \Lambda$ we have  
\[ \dotprod{e_\lambda}{f}_{L^2(\Omega)}= \int_E {e_\lambda(x)} dx
- \int_{E+\tau} {e_\lambda(x)} dx = 0, \]
due to the periodicity of $e_\lambda$. Hence $f$ is orthogonal in $L^2(\Omega)$ to 
all the exponentials $\left\{ e_\lambda \right\}$, $\lambda \in \Lambda$, 
which contradicts the completeness of the system $E(\Lambda)$
in the space $L^2(\Omega)$.
\end{proof}
\subsection{}
Combining \thmref{thmC1} and \lemref{lemC2} we obtain the following: 
\begin{corollary}
	\label{corC3} 
	Let $\Omega\subset\R^d$ be a convex polytope, which is centrally symmetric and has
	centrally symmetric facets. Suppose that $\Omega$ admits a spectrum $\Lambda$ satisfying
	\eqref{eqC1.2} for every facet $F$ of $\Omega$. Then $\Omega+T$ is a tiling,
	and so $\Omega$ can tile by translations. 
\end{corollary}
Moreover, in this case the set $T$ defined by \eqref{eqC1.1} is a lattice in $\R^d$, and 
$\Omega$ tiles face-to-face by translations along the lattice $T$.

\begin{remark*}
The formulation of \corref{corC3} is inspired by \cite[p.\ 568]{IKT03}, where 
the assertion was proved in dimension $d=2$ by directly showing that $\Omega$ must
be either a parallelogram or a  centrally symmetric hexagon.
The proof in arbitrary dimension that we have given above is based 
on different considerations than the one in \cite{IKT03}.
\end{remark*}


\section{Structure of spectrum, I} \label{secD1}

We obtained  in \secref{secC1} a sufficient condition for a spectral convex 
polytope $\Omega$ in $\R^d$ to tile by translations. This condition (\corref{corC3})
requires the existence of a spectrum $\Lam$ admitting a certain structure. 
In the present section we start to develop an approach to analyze the structure of a given
spectrum $\Lam$.

\subsection{}
Let $\Omega\subset\R^d$ be a convex polytope, which is centrally symmetric and has
centrally symmetric facets. We will assume that 
$ \Omega = -\Omega$, that is, $\Omega$ is symmetric about the origin.
Let $F$ be one of the facets of $\Omega$, and assume that
$ F\subset \left\{ x_1 = \tfrac{1}{2} \right\}$,
and that the center of $F$ is the point 
$ \left( \tfrac{1}{2}, 0,0,\dots , 0 \right)$.
\par
These assumptions are made merely for convenience. Later on, we will reduce the general
situation to this more specific one by applying an affine transformation. 
\par
The assumptions
imply that 
\[ F = \left\{ \tfrac{1}{2} \right\}\times \Sigma, \]
where $\Sigma$ is a convex polytope in $\R^{d-1}$ such that 
\[ \Sigma = -\Sigma. \] 
The facet opposite to $F$ is therefore 
\[ -F = \left\{ -\tfrac{1}{2} \right\}\times \Sigma. \]

\subsection{} 
For $\alpha> 0$ we consider the cone 
\begin{equation}
	K(\alpha) := \left\{\xi\in \R^d\; : \; |\xi_j|\le \alpha|\xi_1| \quad (2 \le j \le d)  \right\}.
	\label{eqD1.2}
\end{equation}
\begin{lem}\label{lemD1}
	There is $\alpha = \alpha(\Omega)>0$ such that 
	\begin{equation}
		\pi\xi_1 \hat{\1}_\Omega(\xi)= \sin \pi\xi_1\cdot
		\hat{\1}_\Sigma(\xi_2,\xi_3,\dots,\xi_d) + O( |\xi_1|^{-1} ), \quad
		|\xi_1|\to \infty
		\label{eqD1.1}
	\end{equation}
in the cone $K(\alpha)$.
\end{lem}
\begin{proof}
	By \lemref{lemP4}, if $\alpha$ is sufficiently small then 
	\[ -2\pi i  \xi_1 \hat{\1}_\Omega(\xi) = \hat{\sigma}_F(\xi) -
		\hat{\sigma}_{-F}(\xi) + O( |\xi_1|^{-1} ), \quad
	|\xi_1|\to\infty \]
	in $K(\alpha)$. But we have
\[
				\hat{\sigma}_F(\xi) =  e^{-\pi i \xi_1}\,
				\hat{\1}_{\Sigma}(\xi_2,\xi_3,\dots, \xi_d) \quad \text{and} \quad
				\hat{\sigma}_{-F}(\xi) =  e^{\pi i \xi_1}\,
				\hat{\1}_{\Sigma}(\xi_2,\xi_3,\dots, \xi_d), 
\]
	which yields the conclusion of the lemma. 	
\end{proof}
\subsection{} 
Assume now that we are given a set $\Lambda\subset\R^{d}$ which is a spectrum for $\Omega$. 
To this spectrum $\Lambda$ we associate a set $\Pi\subset \R^{d-1}$ defined as follows: 
$\Pi$ is the set of all points $s\in \R^{d-1}$ such that for every open ball $B$ containing $s$, 
the cylinder $\R\times B$ contains infinitely many points of $\Lambda$.
\par
 If we denote a point in $\R^d$ as $(t,s)\in \R\times \R^{d-1}$, then one can check
that a point $s\in \R^{d-1}$ belongs to $\Pi$ if and only if there is a sequence
$(t_n,s_n)\in \Lambda$ such that 
\begin{equation}
	|t_n|\to \infty, \quad s_n\to s \quad (n\to \infty).
	\label{eqD1.3}
\end{equation}
\par
It is also not difficult to verify that $\Pi$ is a \define{closed} subset of $\R^{d-1}$.
\par
The motivation for introducing the set $\Pi$ is the following observation:
\begin{lem} \label{lemD2}
	For each $s\in \Pi$ there is a (unique) number $0\le \theta(s)<1$ such that
\[ \Lambda\cap \big( \R \times (s+U_{\Sigma}) \big)\subset \big( \Z
	+\theta(s)\big)\times \R^{d-1},\]
 where 
	\[ U_{\Sigma}:= \big\{ \hat{\1}_\Sigma \ne 0 \big\}. \]
	In other  words, if $(t',s')\in \Lambda$ and if $\hat{\1}_{\Sigma}(s'-s)\ne 0$,
	then $t'\in \Z+\theta(s)$. 
\end{lem}
\begin{proof}
	It would be enough to show that if $(t',s')$ and $(t'',s'')$ are two points in $\Lambda\cap
	\big(\R\times(s+U_\Sigma)\big)$, then $ t'' - t' \in \Z$.
	Since $s\in \Pi$ there is a sequence $(t_n,s_n)\in \Lambda$ such that $|t_n|\to
	\infty$, $s_n\to s$. The vectors $(t'-t_n,s'-s_n)$ and $(t''-t_n,s''-s_n)$ belong
	to the set $(\Lambda-\Lambda) \setminus \{0\}$ for all large enough $n$, 
	hence they lie in the zero set of $\hat{\1}_\Omega$. Using
	\lemref{lemD1} it follows that 
	\[
		\sin \pi (t'-t_n)\cdot \hat{\1}_\Sigma (s'-s_n) \to 0, \qquad
		\sin \pi (t''-t_n)\cdot \hat{\1}_\Sigma (s''-s_n) \to 0, 
	\] 
	as $n\to \infty$. 
	Recall that $s'-s$ and $s''-s$ are not in the zero set of $\hat{\1}_\Sigma$. Hence
	$|\hat{\1}_\Sigma(s'-s_n)|$ and $|\hat{\1}_\Sigma(s''-s_n)|$ remain bounded away from zero as
	$n\to \infty$. We conclude that 
	\[ \sin \pi(t'-t_n), \quad \sin \pi (t''-t_n) \]
	both tend to zero as $n\to\infty$, or equivalently, 
	\[ \dist(t'-t_n,\Z), \quad \dist (t''-t_n,\Z) \]
	both tend to zero. But 
	\[ \dist (t''-t',\Z) \le \dist (t'-t_n,\Z)+\dist(t''-t_n,\Z), \]
	which implies that $t''-t'\in \Z$. 
\end{proof}
\begin{corollary}
	Let $s',s''\in \Pi$. If $\theta(s')\ne \theta(s'')$, then
	$\hat{\1}_\Sigma(s''-s')=0$.  
	\label{corD3}
\end{corollary}
\begin{proof}
	Let $(t_n,s_n)\in \Lambda$ be a sequence such that $|t_n|\to \infty$, $s_n\to
	s''$. If $\hat{\1}_\Sigma(s''-s')\ne 0$, then for large enough $n$ we would have
	$\hat{\1}_\Sigma(s_n-s')\ne 0$. By \lemref{lemD2} it follows that $t_n\in
	\Z+\theta(s')$. On the other hand, for all large enough $n$ we also have
	$\hat{\1}_\Sigma(s_n-s'')\ne 0$, since 
	\[ \hat{\1}_\Sigma(0) = |\Sigma|> 0. \]
	Hence, again by \lemref{lemD2}, we have $t_n\in \Z+\theta(s'')$. 
	So we must have $\theta(s')=\theta(s'')$. 
\end{proof}
\subsection{}
\lemref{lemD2} allows us to define an equivalence relation on $\Pi$ by saying that
$s'\sim s''$ if $\theta(s')=\theta(s'')$. It follows from \corref{corD3} that if
$\Pi'$ and $\Pi''$ are two distinct equivalence classes, then 
\[ \Pi''-\Pi' \subset \big\{\hat{\1}_\Sigma = 0\big\}.\]
The set $\{ \hat{\1}_\Sigma = 0 \}$ is disjoint from the open ball of radius
$\chi(\Sigma)>0$ centered at the origin, see \eqref{eqP1.3}. It follows that each equivalence class
is a closed set, and that there can be at most countably many  equivalence classes.
So we may enumerate them as $\Pi_0,\Pi_1,\Pi_2,\dots$ (finitely or infinitely many), and we
denote by $\theta_0,\theta_1,\theta_2,\dots$ respectively the values of the function
$\theta(s)$ on these equivalence classes. 
\subsection{}
To illustrate the construction above, let us consider two representative examples. 
\begin{example} \label{expD4.1}
	Assume that $\Omega$ tiles face-to-face along a lattice $T$ of translation vectors,
	which in this case is given by \eqref{eqC1.1}. Since the facet $F$ has the form
	$F=\left\{ \frac{1}{2} \right\}\times \Sigma$, we have $\tau_F=(1,0,0,\dots, 0)\in T$. 
Let $\Lambda$ be a spectrum of $\Omega$ given by the dual lattice, that is, $\Lambda=T^*$.
Then $\dotprod{\Lambda}{\tau}\subset \Z$ for any $\tau\in T$. In particular this is true
for $\tau=\tau_F$, hence 
\[ \Lambda\subset\Z\times\R^{d-1}. \]
It follows that $\theta(s)=0$ for all $s\in \Pi$. Thus in this case the set $\Pi$ consists of a single
equivalence class, namely $\Pi=\Pi_0$, and we have $\theta_0=0$.
\end{example}

\begin{example} \label{expD4.2}
	Assume that 
	$ \Omega = I \times \Sigma$,
	where $I$ denotes the interval $[-\frac{1}{2}, \frac{1}{2}]$. Then
	$\Omega$ is a prism with base $\Sigma$. Suppose that $\Sigma$ is a spectral set, and let
	$\Gamma \subset \R^{d-1}$ be a spectrum for $\Sigma$. For each $\gamma\in \Gamma$,
	let $\theta(\gamma)$ be an arbitrary real number, $0\le \theta(\gamma)<1$, and define 
	\[ \Lambda := \bigcup_{\gamma\in \Gamma} \big( \Z+\theta(\gamma) \big)\times
	\{ \gamma \}. \]
	It is known, see \cite[Theorem 4]{JorPed99}, that $\Lambda$ is a spectrum for
	$\Omega$. In this case we clearly have $\Pi=\Gamma$, and the numbers
	$\theta(\gamma)$ coincide with the ones given by \lemref{lemD2}. The equivalence
	classes $\Pi_j$ depend on the specific choice of the numbers $\theta(\gamma)$, but in the
	case when all the $\theta(\gamma)$ are distinct, the sets $\Pi_j$ are singletons.
	Observe that we have 
	\[ \Pi_j-\Pi_k \subset \big\{ \hat{\1}_\Sigma = 0 \big\} \quad (k\ne j), \]
	since $\Gamma$ is a spectrum for $\Sigma$. This is in accordance with
	\corref{corD3}.
\end{example}


\section{Structure of spectrum, II} \label{secE1}
In this section we continue to work under the same assumptions as in \secref{secD1}.
Namely, we assume that $\Omega\subset\R^d$ is a convex polytope which is centrally
symmetric, $\Omega=-\Omega$, and has centrally symmetric facets, $F$ is one of the facets
of $\Omega$, and $F=\{\tfrac{1}{2}\}\times \Sigma$ where $\Sigma$ is a convex polytope in $\R^{d-1}$ such
that $\Sigma=-\Sigma$. 
\par
We also assume that $\Lambda$ is a spectrum for $\Omega$, and to
this spectrum $\Lambda$ we associate the set $\Pi\subset\R^{d-1}$ that was defined
in \secref{secD1}.

\subsection{}
From the given spectrum $\Lambda$ one can construct a new spectrum $\Lambda'$ for
$\Omega$, in the following way. Consider the sequence of translates of $\Lam$ given by
\[ \Lambda-k\cdot(1,0,0,\dots,0), \quad k=1,2,3,\dots. \]
Each one of these sets is a spectrum for $\Omega$, and they are uniformly
discrete with the same separation constant. Hence one may extract from this sequence a subsequence 
\[ \Lambda-k_n\cdot (1,0,0,\dots,0), \quad k_n\to\infty, \]
which converges weakly to some set $\Lambda'$, which is also a spectrum for $\Omega$
(see \secref{secLimits}). 
Notice that we do not make any claim concerning the uniqueness of the weak limit
$\Lambda'$, which in general may depend on the particular subsequence that was selected. 
\begin{lem} \label{lemE1}
	We have 
	\begin{equation}
		\Lambda'\subset\bigcup_{j \ge 0} \left(\Z+\theta_j \right)\times \Pi_j.
		\label{eqE1.6}
	\end{equation}
\end{lem}
We remind that by $\theta_j$ $(j \ge 0)$ we denote the distinct values attained by
the function $\theta(s)$ defined on $\Pi$, given in \lemref{lemD2}, and 
	\begin{equation}
		\label{eqE1.1.2}
		\Pi_j=\left\{ s\in \Pi\;:\; \theta(s)=\theta_j \right\}.
	\end{equation}
Recall also that according to \corref{corD3} we have
\begin{equation}
\Pi_k-\Pi_j\subset \big\{ \hat{\1}_\Sigma=0 \big\} \quad (j\ne k),
\label{eqE1.1}
\end{equation}
hence \lemref{lemE1} reveals a certain structure satisfied by the new spectrum $\Lambda'$.
\begin{proof}[Proof of \lemref{lemE1}]
	The claim is equivalent to the statement that for every $(t',s')\in \Lambda'$ we
	have $s'\in\Pi$ and $t'\in\Z+\theta(s')$. Let therefore $(t',s')\in\Lambda'$.
	Since $\Lambda'$ is the weak limit of the sequence
	$\Lambda-k_n\cdot(1,0,0,\dots,0)$, there exist $(t_n,s_n)\in \Lambda$ such that 
	\[ (t_n-k_n,s_n)\to (t',s'), \quad n\to\infty. \]
	Hence $s_n\to s'$, and $t_n\to \infty$ since $k_n\to\infty$. This implies that
	$s'\in\Pi$. For all sufficiently large $n$ we have $\hat{\1}_\Sigma(s_n-s')\ne 0$,
	thus by \lemref{lemD2} we have $t_n\in\Z+\theta(s')$. Since $t_n-k_n\to t'$ and
	the $k_n$ are integers, this implies that also $t'\in \Z+\theta(s')$. 
\end{proof}
\subsection{}
Given a point  $(t_0,s_0)$ in $\R\times\R^{d-1}$, we associate with it a function $f$ defined by 
\begin{equation}
	f(x,y) := \1_{I}(x)e^{2\pi i t_0 x} \, \1_{\Sigma}(y)e^{2\pi i \dotprod{s_0}{y}},
\quad (x,y)\in \R\times\R^{d-1}
	\label{eqE6.1}
\end{equation}
where $I$ denotes again the  interval $[-\tfrac{1}{2},\tfrac{1}{2}]$. 
Notice that the function $f$ is supported by the prism $I\times\Sigma$.
This prism is contained in $\Omega$, since
$ \left\{ \tfrac{1}{2} \right\}\times \Sigma$ and $\left\{ -\tfrac{1}{2}
\right\}\times\Sigma$ are facets of $\Omega$, and $\Omega$ is convex.
Hence $f$ is also supported by $\Omega$.
\par
It follows from the definition \eqref{eqE6.1} of $f$ that its Fourier transform is given by
\begin{equation}
	\label{eqE6.1.2}
\hat{f}(t,s)=\hat{\1}_I(t-t_0)\hat{\1}_\Sigma(s-s_0), \quad (t,s)\in
		\R\times \R^{d-1}.
\end{equation}
\par
Using the function $f$ thus defined, we can prove a result similar to \lemref{lemE1} but
which is concerned with the originally given spectrum $\Lambda$.
However, the conclusion is somewhat weaker, as the right hand side of
\eqref{eqE1.6} is replaced by a larger set:
\begin{lem} \label{lemE3}
We have
\[ \Lambda \subset \bigcup_{j \ge 0} \left(\Z+\theta_j\right)\times (\Pi_j+U_\Sigma), \]
where, as before, we let 
\[ U_\Sigma = \big\{ \hat{\1}_\Sigma \ne 0 \big\}. \]
\end{lem}
\begin{proof}
	By \lemref{lemD2} we have 
	\[ \Lambda \cap \big( \R\times (\Pi_j+U_\Sigma)\big) \subset (\Z+\theta_j)\times
	(\Pi_j+U_\Sigma) \]
	for every $j$. Hence, to prove the claim it would be enough to show that the sets 
	$\Pi_j+U_\Sigma$  cover the whole $\R^{d-1}$. 	Suppose to the contrary that
	there is a point $s_0\in\R^{d-1}$ which lies outside all the sets $\Pi_j+U_\Sigma$.
	Since $U_\Sigma = -U_\Sigma$, this means that 
	\[ \hat{\1}_\Sigma(s-s_0)= 0, \quad s\in\Pi. \]
	Let $t_0$ be an arbitrary real number, and consider the function $f$ defined by
	\eqref{eqE6.1}. Then $f$ is supported by $\Omega$, and by \eqref{eqE6.1.2}
	its Fourier transform $\hat{f}$ vanishes on $\R\times\Pi$. In particular we have
	$\hat{f}(\lambda)=0$ for all $\lambda\in\Lambda'$, due to \lemref{lemE1}. That is, 
	\[ \dotprod{f}{e_\lambda}_{L^2(\Omega)} = \hat{f}(\lambda)=0, \quad \lambda\in
	\Lambda'. \]
	Hence $f$ is orthogonal in $L^2(\Omega)$ to all the exponentials $\left\{ e_\lambda \right\}$,
	$\lambda \in \Lambda'$, which contradicts the completeness of the system $E(\Lambda')$
	in the space $L^2(\Omega)$. This  concludes the proof.
\end{proof}
\begin{corollary}
	\label{corE7}
	Assume that the function $\theta(s)$ is constant on $\Pi$. Then  
	\begin{equation}
		\label{eqE7.1}
		\Lambda-\Lambda \subset \Z\times \R^{d-1}.
	\end{equation}
\end{corollary}
\begin{proof}
	It is assumed that $\Pi=\Pi_0$ and $\theta(s)=\theta_0$ for all $s\in\Pi$. Hence by
	\lemref{lemE3}, the set $\Lam$ is contained in 
$ (\Z+\theta_0)\times(\Pi_0+U_\Sigma)$, which implies \eqref{eqE7.1}.
\end{proof} 
\subsection{}
\corref{corE7} is an important point in our approach to the proof that $\Omega$ can tile by
translations. Let us clarify its role. Recall that a sufficient condition for $\Omega$ to tile
was given by \corref{corC3}, namely, it is enough to know that the spectrum $\Lambda$
satisfies condition \eqref{eqC1.2} for every facet $F$ of $\Omega$. For the facet
$F=\{\tfrac{1}{2}\}\times \Sigma$ we have $\tau_F=(1,0,0,\dots,0)$, hence for this facet
the condition \eqref{eqC1.2} is the same as \eqref{eqE7.1}. It thus follows from
\corref{corE7} that in order to establish \eqref{eqC1.2} for the facet
$F=\{\tfrac{1}{2}\}\times\Sigma$, it would be sufficient to prove that the function $\theta(s)$ is
constant on $\Pi$.


\section{Spectral convex polygons tile the plane} \label{secE2}
\subsection{}
In this section we will demonstrate how the tools developed so far can be useful in our problem, 
by showing that at this point they already enable us to give an alternative proof of the
following result in dimension $d=2$:
\begin{thm}[Iosevich, Katz and Tao \cite{IKT03}]
\label{thmE8}
	Let $\Omega$ be a convex polygon in $\R^2$. If $\Omega$ is spectral, then $\Omega$ tiles
	by translations. 
\end{thm}
We remark that the paper \cite{IKT03} actually contains a proof of a more general result, which yields the
same conclusion for any convex body $\Omega \subset \R^2$ (not assumed a priori to be a polygon). 

\subsection{}
In order to prove \thmref{thmE8}, we now restrict ourselves to dimension $d=2$. Let
$\Omega$ be a convex polygon in $\R^2$. Assume that $\Omega$ is spectral, and let
$\Lambda$ be a spectrum for $\Omega$. We must prove that $\Omega$ tiles by translations.
This is obvious if $\Omega$ is a parallelogram, so in what follows we will assume 
that $\Omega$ is not a parallelogram. 
\par
By \thmref{thmA1} the polygon $\Omega$ is centrally symmetric, and since the facets
of $\Omega$ are line segments, then automatically also all the facets of $\Omega$ are 
centrally symmetric.
\begin{lem} \label{lemE9}
	Let $\Omega$ be a convex, centrally symmetric polygon in $\R^2$, and assume that
	$\Omega$ is not a parallelogram. If $\Lambda$ is a spectrum of $\Omega$, then 
	\begin{equation}
		\dotprod{\Lambda-\Lambda}{\tau_F}\subset \Z
		\label{eqE8.1}
	\end{equation}
	for every facet $F$ of $\Omega$. 
\end{lem}
\thmref{thmE8} follows immediately from a combination of \lemref{lemE9} and \corref{corC3}. Hence, it
only remains to prove the lemma. 
\par
\lemref{lemE9} was proved in
\cite[Proposition 3.1]{IKT03}, and was also used there
to  deduce that $\Omega$ tiles by translations. However, both our proof of
\lemref{lemE9}, and the argument we use to deduce \thmref{thmE8} from \lemref{lemE9}, are different
from those in \cite{IKT03}. 
\subsection{}
Now we give our proof of \lemref{lemE9}. 
\begin{proof}[Proof of \lemref{lemE9}] 
	Let $F$ be a facet of $\Omega$. We must show that if $\Lambda$ is a spectrum of
	$\Omega$, then it satisfies condition \eqref{eqE8.1}. By applying an affine
	transformation we may assume that $\Omega$ is symmetric about the origin,
	$\Omega=-\Omega$, and that $F=\{\tfrac{1}{2}\}\times I$, where $I$ is the
	interval $[-\tfrac{1}{2},\tfrac{1}{2}]$. Hence we have $\Sigma=I$,
	$\tau_F=(1,0)$, and condition \eqref{eqE8.1} becomes 
	\begin{equation}
		 \Lambda-\Lambda\subset\Z\times\R.
		 \label{eqE8.2}
	\end{equation}
	Let $\Pi\subset \R$ be the set associated to the spectrum $\Lambda$ defined as in
	Section \ref{secD1}, and $\theta(s)$ be the function on $\Pi$ given by
	\lemref{lemD2}. By \corref{corE7}, to establish \eqref{eqE8.2} it would be enough to show
	that $\theta(s)$ is constant on $\Pi$. 
\par
	Let us first consider the case when 
	\begin{equation}
		 \Pi-\Pi\subset\Z.
		 \label{eqE8.4}
	\end{equation}
	We will show that in this case we must have $\Omega=I\times I$, that is, $\Omega$
	is the unit cube, which is not possible as we have assumed that $\Omega$ is not a
	parallelogram. Indeed, suppose that \eqref{eqE8.4} holds, and let $\Lambda'$ be
	the spectrum constructed from $\Lambda$ in Section \ref{secE1}. Fix a point
	$\lambda_0=(t_0,s_0)\in\Lambda'$. It follows from \lemref{lemE1} and \eqref{eqE8.4}
	that if $\lambda'=(t',s')$ is any point in $\Lambda'$ other than $\lambda_0$,
	then at least one of the numbers $t'-t_0$ and $s'-s_0$ must be in $\Z\setminus \{0\}$.
	Now consider the function $f$ defined by
	\eqref{eqE6.1}. This function is supported by $\Omega$, and by \eqref{eqE6.1.2}
	its Fourier transform  $\hat{f}$ vanishes
	on all the points of $\Lambda'$ except from $\lambda_0$, since $\hat{\1}_I$ vanishes
	on $\Z \setminus \{0\}$. Hence $f$ is orthogonal  in $L^2(\Omega)$ to
	all the exponentials $\{e_\lambda\}$, $\lambda\in \Lambda'\setminus
	\{\lambda_0\}$. Since the system $E(\Lam')$ is orthogonal and complete in $L^2(\Omega)$,
	this implies that $f$ must coincide a.e.\ on $\Omega$ with a
	constant (non-zero) multiple of $e_{\lambda_0}$. In particular, $f$ cannot vanish
	on any subset of $\Omega$ of positive measure. On the other hand, by the
	definition of $f$ it does vanish on	$\Omega\setminus (I\times I)$. This is possible only if $\Omega=I\times
	I$.
\par
	We thus conclude that \eqref{eqE8.4} is not possible, so we must have
	\begin{equation}
		 \Pi-\Pi\not\subset\Z.
		 \label{eqE8.3}
	\end{equation}
	Let us then show that $\theta(s)$ is a constant function on $\Pi$. Indeed, due to \eqref{eqE8.3}
	there exist $s',s''\in \Pi$ such that $s'' - s' \notin \Z$. Since $\big\{\hat{\1}_I=0\big\} = \Z \setminus \{0\}$,
	\corref{corD3} implies that $\theta(s')=\theta(s'')$. Observe that for any $s \in \Pi$ we must have
	$s - s' \notin \Z$ or $s - s'' \notin \Z$, and in either case we obtain, again
	by \corref{corD3}, that $\theta(s)=\theta(s')=\theta(s'')$. This shows that $\theta(s)$ 
	must be a constant function on $\Pi$. \lemref{lemE9} is therefore proved. 
\end{proof}


\section{Prisms and cylindric sets} \label{secF1}
\subsection{}
The proof presented in \secref{secE2} that a spectral convex polygon in the plane
$\R^2$ can tile by translations,  eventually relied on showing that the function
$\theta(s)$ is constant on the set $\Pi$. In order to show this we had to exclude the case
when $\Omega$ is a parallelogram, but since a parallelogram automatically tiles by translations, this loss of
generality was innocuous in the proof.
\par
 In dimension $d=3$, however, the situation is more
complicated. Even if we exclude the case when $\Omega$ is a parallelepiped, one still
cannot expect to be able to prove that $\theta(s)$ is a constant function on $\Pi$.
Indeed, we have seen in \exampref{expD4.2} above that if $\Omega$ is a prism whose
base is a spectral set, then the
function $\theta(s)$ may attain countably many, arbitrary distinct values. Hence, the role
of the parallelogram in dimension $d=2$ will be played not by the parallelepiped,
but by the prism, in dimension $d=3$. 
\par
We remind the reader that by a \define{prism} in $\R^d$ one means a
polytope $\Omega$ which can be expressed as the Minkowski sum of a $(d-1)$-dimensional
polytope and a line segment. 
\par
Notice, however, that while a parallelogram in
$\R^2$ automatically tiles by translations, this is not so for a prism in $\R^3$. Hence it
is yet required to prove -- necessarily by a different method -- that a spectral convex prism in
$\R^3$ can tile by translations. 
\par
Let us formulate this result explicitly:
\begin{thm} \label{thmF1}
	Let $\Omega$ be a convex prism in $\R^3$. If $\Omega$ is spectral, then it tiles
	by translations. 
\end{thm}
\subsection{}
A bounded, measurable set $\Omega\subset\R^d$ $(d \geq 2)$ will be called a \define{cylindric set} if it
has the form $\Omega=I\times\Sigma$, where $I$ is an interval in $\R$, and $\Sigma$ is a
measurable set in $\R^{d-1}$. In this case, the set $\Sigma$ will be called the \define{base} of
the cylindric set $\Omega$.
\par
 If the base $\Sigma$ is a convex polytope in $\R^{d-1}$, then the set
$\Omega=I\times\Sigma$ is a convex prism. Conversely, any convex prism in $\R^d$ is the
affine image of some set of the form $I\times\Sigma$, where $I$ is an interval and
$\Sigma$ is a convex polytope in $\R^{d-1}$. 
\par
We will deduce \thmref{thmF1} from the
following result, proved in our paper \cite{GriLev16a}. The result is valid in all
dimensions $d\ge 2$ (not just $d=3$). 
\begin{thm}[\cite{GriLev16a}] \label{thmF2}
	A cylindric set $\Omega=I\times \Sigma$ is spectral
	(as a set in $\R^d$) if and only if its base $\Sigma$ is a spectral set (as a set in
	$\R^{d-1}$). 
\end{thm}
This result thus provides a characterization of the cylindric spectral sets $\Omega$ in terms of the
spectrality of their base $\Sigma$. 
\par
 The ``if'' part of \thmref{thmF2} is obvious. Suppose for simplicity that
$I=[-\tfrac{1}{2},\tfrac{1}{2}]$. If $\Gamma\subset\R^{d-1}$ is a spectrum for $\Sigma$,
then it is easy to check that $\Z\times\Gamma$ is a spectrum for $\Omega$, hence
$\Omega$  is spectral. 
\par
On the other hand, the converse, ``only if'' part of the theorem (which is what we
shall need for our present goal), is non-trivial. Roughly speaking, 
the difficulty lies in that knowing $\Omega$ to
have a spectrum $\Lambda$ in no way implies that $\Lambda$ has a product structure as
$\Z\times \Gamma$. In particular, we do not have any obvious candidate for a set
$\Gamma\subset\R^{d-1}$ that might serve as a spectrum for $\Sigma$.

\begin{remark*}
In \cite{GriLev16a} we also gave a similar characterization 
of the cylindric sets $\Omega$ in $\R^d$ which can tile the space by translations.
Namely, it was proved there that a cylindric set $\Omega=I\times \Sigma$
tiles if and only if its base $\Sigma$ tiles.
\end{remark*}

\subsection{}
\thmref{thmF1} can now be obtained by a combination of \thmref{thmF2} and the result from
\cite{IKT03} that a spectral convex polygon in $\R^2$ can tile by translations, namely,
\thmref{thmE8} (for which we have provided an independent proof in \secref{secE2}).
\begin{proof}[Proof of \thmref{thmF1}]
	By applying an affine transformation we can assume that $\Omega=I\times \Sigma$,
	where $I$ is the interval $[-\tfrac{1}{2},\tfrac{1}{2}]$ and $\Sigma$ is a convex
	polygon in $\R^2$. Since $\Omega$ is spectral, it follows by \thmref{thmF2} that
	also $\Sigma$ is spectral. Hence by \thmref{thmE8}, $\Sigma$ tiles by
	translations, so there is a set $\Gamma\subset \R^2$ such that $\Sigma+\Gamma$ is
	a tiling of $\R^2$. It is then clear that $\Omega$ tiles $\R^3$ with the
	translation set $\Z\times\Gamma$, and this completes the proof. 
\end{proof}


\section{Prisms and zonotopes}\label{secF2}

In \secref{secF1} we explained why the case when the convex polytope $\Omega \subset \R^3$
is a prism requires a special treatment in our approach. In this case we obtained 
a complete solution to our problem, namely, it was proved that if a convex prism
in $\R^3$ is a spectral set, then it tiles by translations (\thmref{thmF1}). 
Hence, in what follows we will be mainly interested in
the case when $\Omega$ is not a prism. The goal of the present section is
to point out some geometric properties of such an $\Omega$, that will be useful
in the analysis of the spectrum later on.

\subsection{}
Let $\Omega \subset \R^3$ be a convex polytope, centrally symmetric and
with centrally symmetric facets. 
Let $F$ be a facet of $\Omega$, and $F'$ be the opposite facet. Recall that by the central symmetry,
$F'$ is a translate of $F$, and that we have denoted by $\tau_F$ the translation vector in
$\R^3$ which carries $F'$ onto $F$, that is, $F=F'+\tau_F$. 
\par
Suppose  now that $A$ is a subfacet of
$F$. Then $A$ is the image under the translation by $\tau_F$ of a subfacet $A'$ of
$F'$, that is, $A=A'+\tau_F$. We denote by $H_{F,A}$ the hyperplane which contains the
subfacets $A$ and $A'$. 
\begin{lem}\label{lemF3}
	If $\Omega$ is not a prism, then for any facet $F$ of $\Omega$ there is a subfacet
	$A$ such that $\interior(\Omega)$ intersects each one of the two open half-spaces bounded by
	$H_{F,A}$. 
\end{lem}
\begin{proof}
	Let $F$ be a facet of $\Omega$. By applying an affine transformation we may assume
	that
	\[ \Omega=-\Omega, \quad F=\left\{ \tfrac{1}{2} \right\}\times\Sigma, \quad F'=\left\{
	-\tfrac{1}{2} \right\}\times\Sigma, \]
	where $\Sigma$ is a convex polygon in $\R^2$ such that $\Sigma=-\Sigma$. 
	Suppose to the contrary that for any subfacet $A$ of $F$, $\interior(\Omega)$
	entirely lies within one of the open half-spaces bounded by $H_{F,A}$. The
	intersection of the closures of all these half-spaces with the set $I\times\R^2$,
	where $I=[-\tfrac{1}{2},\tfrac{1}{2}]$, is equal to $I\times\Sigma$. Hence
	$\Omega$ is contained in $I\times\Sigma$. But $\Omega$ also contains
	$I\times\Sigma$, since $I \times \Sigma$ is the convex hull of the facets $F$ and $F'$.
	We conclude that $\Omega=I\times \Sigma$, which is not
	possible unless $\Omega$ is a prism. This contradiction ends the proof. 
\end{proof}
\subsection{}
By a \define{zonotope} in $\R^d$ one means a polytope which can be represented as the
Minkowski sum of a finite number of line segments. A zonotope is a convex, centrally
symmetric polytope, and all its facets are also zonotopes. In particular, all  the facets
of a zonotope are also centrally symmetric.
\par
 It is known, see e.g.\ \cite[Theorem 3.5.1]{Sch93}, that
in dimension $d=3$, a convex polytope which has centrally
symmetric facets must be a zonotope.
\par Remark, by the way, that this is not true in dimensions $d\ge 4$.
 A well-known example is the $24$-cell in $\R^4$, a convex polytope which tiles by
translations, and hence it is centrally symmetric and has centrally symmetric facets, but which
is not a zonotope. 
\subsection{} 
Let again $\Omega\subset \R^3$ be a convex polytope, centrally symmetric and with centrally
symmetric facets (and hence a zonotope). Let $F$ be a facet of $\Omega$, and $A, B$ be two
parallel subfacets of $F$. Let $F'$ and $A', B'$ be the facet and its two subfacets which
are carried onto $F$ and $A,B$ respectively by the translation vector $\tau_F$. We
denote by $S_{F,A,B}$ the closed slab which lies between the two parallel hyperplanes
$H_{F,A}$ and $H_{F,B}$. 
\begin{lem} \label{lemF4}
	Assume that the intersection of $\Omega$ and $S_{F,A,B}$ coincides with
	the convex hull of the facets $F$ and $F'$. Then $\Omega$ is a prism.
\end{lem}
\begin{proof}
	By applying an affine transformation we may assume that 
\[ \Omega=-\Omega, \quad
F\subset \{x_1=\tfrac{1}{2}\},\]
$F$ is symmetric about the point 	$(\tfrac{1}{2},0,0)$,
\[ A=\{\tfrac{1}{2}\}\times\{\tfrac{1}{2}\}\times I, \quad
B=\{\tfrac{1}{2}\}\times\{-\tfrac{1}{2}\}\times I, \]
 where $I$ denotes as usual the interval $[-\tfrac{1}{2},\tfrac{1}{2}]$. 
Hence $F=\{\tfrac{1}{2}\} \times \Sigma$, where
		$\Sigma$ is a convex polygon in $\R^2$ such that $\Sigma=-\Sigma$, and
such that
		$\{\tfrac{1}{2}\}\times I$, $\{-\tfrac{1}{2}\}\times I$ are facets of
		$\Sigma$. 
\par
	The assumption in the lemma thus means that 
		\begin{equation}
			\Omega\cap (\R\times I \times \R) = I\times \Sigma.
			\label{eqF4.1}
		\end{equation}
\par
		Since $\Omega$ is a zonotope, it can be represented as the Minkowski sum of
		several line segments
		$ S_1, S_2,\dots , S_n$. Thus we have $\Omega=S_1+S_2+\dots+S_n$.
		As $\Omega$ is symmetric about the origin, we can assume that the same is
		true for each line segment $S_j$, that is, $S_j=-S_j$. We can also assume that no two
		of the segments $S_j$ are parallel. 
\par
		Now we consider two distinct cases
		separately. Let us first consider the case when $\Sigma$ is not the cube
		$I\times I$. In this case there must exist at least one vertex $v$ of
		$\Sigma$, which belongs to $\interior(I\times \R)$. Hence $I\times\{v\}$
		is a subfacet of $I\times \Sigma$. By \eqref{eqF4.1} it follows that
		$I\times\{v\}$ is also a subfacet of $\Omega$. Each subfacet of $\Omega$
		is a translate of one of the $S_j$'s (see, for example, \cite{McMul71}).
		Hence one of the line segments, say $S_1$, must be equal to $I\times
		\{0\}\times\{0\}$. It then follows that all the other line segments
		$S_2,\dots , S_n$ must lie in $\{0\}\times \R\times\R$. Indeed, if this is
		not true for some $S_j$, then $S_1+S_j$ is not contained in
		$I\times\R\times\R$. But $S_1+S_j$ is contained in $\Omega$, and
		$\Omega$ is contained in $I\times\R\times\R$, so this is not possible. 
		Hence all the segments $S_2,\dots,S_n$ lie in $\{0\}\times\R\times\R$. It
		follows that $S_2+\dots+S_n=\{0\}\times \Sigma$, and $\Omega=I\times
		\Sigma$. This shows that $\Omega$ must be a prism.
\par
		 Now we consider the remaining case, namely, when $\Sigma  = I\times I$. In this
		case, the assumption \eqref{eqF4.1} becomes 
		\begin{equation}
			\label{eqF4.2}
			\Omega\cap (\R\times I\times\R) = I\times I\times I. 
		\end{equation}
		Hence $\R\times\R\times\{\tfrac{1}{2}\}$ and
		$\R\times\R\times\{-\tfrac{1}{2}\}$ are supporting hyperplanes of
		$\Omega$, and thus $\Omega\subset \R\times\R\times I$. Since
		$A=\{\tfrac{1}{2}\}\times\{\tfrac{1}{2}\}\times I$ is a subfacet of
		$\Omega$, then, as before, one of the line segments, say again $S_1$, must
		be equal to $\{0\}\times\{0\}\times I$. It then follows that all the other
		line segments $S_2,\dots, S_n$ must lie in $\R\times\R\times\{0\}$,
		since if not, then as before this would contradict the fact that
		$\Omega\subset\R\times\R\times I$. Hence $S_2+\dots+S_n=P\times \{0\}$ for
		a certain convex polygon $P\subset\R^2$, and $\Omega=P\times I$. Again we
		obtain that $\Omega$ must be a prism, so this proves the lemma.
\end{proof}


\section{Structure of spectrum, III} \label{secG1}

In this section our goal is to relate the geometric observations  made in
\secref{secF2} to the spectrality problem for convex polytopes in dimension $d=3$. 
More specifically, we will see how one can use the assumption that $\Omega$ is not a prism,
in order to obtain new information on the
structure of the spectrum $\Lam$.

\subsection{}
Let $\Omega \subset \R^3$ be a convex polytope, centrally symmetric and
with centrally symmetric facets. 
Assume, as before, that $\Omega=-\Omega$, that is, $\Omega$ is symmetric 
about the origin, $F$ is a facet of $\Omega$ contained in 
$\left\{ x_1 = \tfrac{1}{2} \right\}$, and that the center of $F$ is the point 
$ \left( \tfrac{1}{2}, 0,0 \right)$. Hence $F = \left\{ \tfrac{1}{2} \right\}\times \Sigma$,
where $\Sigma$ is a convex polygon in $\R^2$ such that $\Sigma = -\Sigma$.
\par
Suppose also that $\Lambda$ is a spectrum for $\Omega$. Let $\Pi\subset\R^2$ be
the set associated to the spectrum $\Lambda$ defined as in Section \ref{secD1}, and
$\theta(s)$ be the function on $\Pi$ given by \lemref{lemD2}. We also let $\Lambda'$ be the
new spectrum constructed from $\Lambda$ in Section \ref{secE1}. 
\par
 Recall that to each point $(t_0,s_0)\in \R\times\R^2$ we have
associated a function $f$, supported by $\Omega$, 
defined by \eqref{eqE6.1}. As an element of $L^2(\Omega)$ this function
$f$ admits  a Fourier expansion with respect to the spectrum $\Lam'$, given by
\begin{equation}
	\label{eqE4.1}
	f=\frac{1}{|\Omega|}\sum_{\lambda\in\Lambda'} \hat{f}(\lambda)e_\lambda.
\end{equation}
By \lemref{lemPL6} the series on the right hand side of
\eqref{eqE4.1} converges in $L^2$ on any bounded set to a measurable function
$\tilde{f}$ on $\R^3$, and $f$ coincides with $\tilde{f}$ a.e.\ on
$\Omega$.
\par
We now observe that for certain values of $(t_0,s_0)$, the Fourier expansion of $f$ with respect to
the spectrum $\Lambda'$ consists of exceptionally few terms:
\begin{lem} \label{lemE4}
	Let $(t_0,s_0)$ be a point belonging to $(\Z+\theta_j)\times\Pi_j$ for some $j$,
	and let $f$ be the function defined by \eqref{eqE6.1}. Then the Fourier expansion
	\eqref{eqE4.1} of $f$ with respect to the spectrum $\Lam'$ consists only of terms corresponding to
	$\lambda\in \Lambda'\cap \left( \{t_0\}\times\Pi_j \right)$.
\end{lem}
In other words, all the coefficients $\hat{f}(\lambda)$ in the expansion \eqref{eqE4.1} must
vanish except for  possibly those which correspond to $\lambda=(t,s)\in \Lambda'$ such that
$t=t_0$ and $s\in \Pi_j$.
\begin{proof}[Proof of \lemref{lemE4}]
	If $(t,s)\in \Lambda'$, then by \lemref{lemE1} there is $k$ such that $t\in
	\Z+\theta_k$ and $s\in\Pi_k$. If $k\ne j$ then $\hat{\1}_\Sigma(s-s_0)=0$ due to
	\eqref{eqE1.1}, and it follows from \eqref{eqE6.1.2}
	that $\hat{f}(t,s)=0$. If $k=j$ then both $t_0$ and $t$ belong to $\Z+\theta_j$,
	hence $t-t_0$ is an integer. Since $\hat{\1}_I$ vanishes on $\Z\setminus \{0\}$, it
	follows again by \eqref{eqE6.1.2} that $\hat{f}(t,s)=0$ unless
	$t=t_0$. This shows that in the series  \eqref{eqE4.1} the non-zero coefficients
	can only correspond to $\lambda=(t,s)$ such that $t=t_0$ and $s\in \Pi_j$. This
	proves the claim.
\end{proof}
\begin{remark*}
It may be interesting to notice that \lemref{lemE4} implies that $\Lambda'$ must contain
points from \emph{each one} of the sets $\{t_0\}\times \Pi_j$,
where $t_0$ goes through the elements of $\Z+\theta_j$. 
\end{remark*}
\subsection{}
Now suppose that $\Omega$ is \emph{not a prism}. Then by
\lemref{lemF3} there is a subfacet $A$ of $F$ such that $\interior(\Omega)$
intersects each one of the two open half-spaces bounded by the  hyperplane
$H_{F,A}$. Let us assume, for simplicity, that this subfacet is
$A=\{\tfrac{1}{2}\}\times\{\tfrac{1}{2}\}\times I$, where
$I=[-\tfrac{1}{2},\tfrac{1}{2}]$
(later on, the general situation will be reduced to this case by applying an affine transformation).
Thus $\{\tfrac{1}{2}\}\times I$  is a facet of the convex polygon $\Sigma$.
\par
We can now use \lemref{lemE4} to obtain some additional information on the structure of the
components $\Pi_j$ of the set $\Pi$.
\begin{lem}\label{lemG3}
	For each $j$ we have 
	\begin{equation}
		 \Pi_j-\Pi_j\not\subset \Z\times \R.
		 \label{eqG3.1}
	\end{equation}
\end{lem}
\begin{proof}
	Suppose that \eqref{eqG3.1} is not true for some $j$. By translating the spectrum
	$\Lambda$ we can assume that $\Pi_j$ contains the origin, and hence 
	\begin{equation}
		 \Pi_j\subset \Z\times\R.
		 \label{eqG3.2}
	\end{equation}
\par
	Choose a point $(t_0,s_0)\in (\Z+\theta_j)\times \Pi_j$, and let $f$ be the
	function associated to this point defined by \eqref{eqE6.1}. By \lemref{lemE4} and
	due to \eqref{eqG3.2}, the Fourier expansion of $f$ with respect to $\Lambda'$
	consists only of exponentials $e_\lambda$ such that $\lambda\in \Lambda'\cap
	(\R\times\Z\times\R)$. It follows (\lemref{lemPL6}) that the right-hand side of \eqref{eqE4.1} is a
	function $\tilde{f}$ on $\R^3$ which is periodic with respect to the vector $(0,1,0)$, and
	$f$ coincides with $\tilde{f}$ a.e.\ on $\Omega$. 
\par
	Recall that we have chosen the subfacet $A$ of $F$ (using \lemref{lemF3}) such that
	$\interior(\Omega)$ intersects each one of the two open half-spaces bounded by the
	hyperplane $H_{F,A}$. Since it was assumed that 
	$ A=\left\{ \tfrac{1}{2} \right\}\times\left\{ \tfrac{1}{2} \right\}\times I$,
	this means that 
	$ H_{F,A} = \left\{ x_2=\tfrac{1}{2} \right\}$,
	and hence 
	\begin{equation}
		 \Omega \not\subset  \left\{ x_2 \leq \tfrac{1}{2} \right\}.
		 \label{eqG3.3}
	\end{equation}
\par
	Recall also that $F=\{\tfrac{1}{2}\}\times\Sigma$, where $\Sigma$ is a convex
	polygon in $\R^2$, $\Sigma=-\Sigma$, and $\{ \tfrac{1}{2} \}\times I$ is a face of
	$\Sigma$. By convexity, $\Sigma$ contains the unit square $I\times I$, and hence
	$I\times \Sigma$ contains the unit cube $I\times I\times I$. Thus
	$|\tilde{f}|= |f| =1$ a.e.\ on $I\times I\times I$. By the periodicity of $\tilde{f}$ this implies that
	$|\tilde{f}|=1$ a.e.\ on $I\times \R\times I$. 
	In particular, $|\tilde{f}|=1$ a.e.\ on the set
	\begin{equation}
		 \label{eqG3.4}
	 \Omega \cap (I \times (\R \setminus I) \times I). 
	\end{equation}
	On the other hand, the set \eqref{eqG3.4} is disjoint from $I \times \Sigma$, hence
	$|f|=0$ on this set. It follows that the set \eqref{eqG3.4} cannot have positive measure, and therefore
	\[ \Omega\cap (I\times \R\times I) = I\times I\times I. \]
	This implies that $\{x_2=\tfrac{1}{2}\}$ is a supporting hyperplane of $\Omega$, which 
	contradicts \eqref{eqG3.3}.
\end{proof}
\begin{lem}\label{lemG4}
	For each $j$ we have 
	\begin{equation}
		 \Pi_j-\Pi_j\not\subset \R\times\Z.
		 \label{eqG4.2}
	\end{equation}
\end{lem}
\begin{proof}
	We argue in a way similar to the proof of the previous lemma. If \eqref{eqG4.2} is violated for some $j$, then by
	translating $\Lambda$ we can assume that 
	\begin{equation}
	 \Pi_j\subset \R\times \Z. 
		 \label{eqG4.3}
	\end{equation}
	Hence, choosing a point $(t_0,s_0)\in (\Z+\theta_j)\times \Pi_j$, the corresponding function $f$ defined by \eqref{eqE6.1} 
	coincides a.e.\ on $\Omega$ with a function $\tilde{f}$ on $\R^3$, which by \eqref{eqG4.3} and \lemref{lemE4}
	is periodic with respect to the vector $(0,0,1)$. 
\par
	Since we have
	$|\tilde{f}|=|f|=1$ a.e.\ on $I\times I\times I$, the periodicity of $\tilde{f}$ implies that
	$|\tilde{f}|=1$ a.e.\ on $I\times I\times\R$.
	In particular, $|\tilde{f}|=1$ a.e.\ on the set
	\begin{equation}
		 \label{eqG4.1}
	 \Omega \cap (I \times ((I \times \R) \setminus \Sigma)).
	\end{equation}
	But since this set is disjoint from $I \times \Sigma$, we have	$|f|=0$ on the set \eqref{eqG4.1}.
	So the set \eqref{eqG4.1} cannot have positive measure, and therefore
	\[ \Omega \cap (I\times I\times \R) = I\times \Sigma. \]
	By \lemref{lemF4} this is possible only if $\Omega$ is a prism, so this concludes the proof.
\end{proof}
\begin{lem} \label{lemG5}
	Let $X$ be a subset of an abelian group $G$, and let $H_1$ and $H_2$ be two
	subgroups of $G$. Assume that 
	\begin{equation}
		 X-X\subset H_1\cup H_2.
		 \label{eqG5.1}
	\end{equation}
	Then $X-X\subset H_1$ or $X-X\subset H_2$. 
\end{lem}
\begin{proof}
	Suppose that $X-X\not\subset H_1$, so there exist $x,y\in X$ such that
	$x-y\not\in H_1$. Then by \eqref{eqG5.1} we have $x-y\in H_2$. The property
	$x-y\not\in H_1$ implies that
	for each $z\in X$ we must have $z-x\notin H_1$ or $z-y\notin H_1$. But
	in either case,  it follows from \eqref{eqG5.1} that $z \in x + H_2 = y+ H_2$, so we
	conclude that $X \subset x + H_2 = y+ H_2$. Thus $X-X\subset H_2$. 
\end{proof}
\begin{corollary} \label{corG6}
	 For each $j$ we have 
	 \begin{equation}
		 \Pi_j-\Pi_j \not\subset (\Z\times \R) \cup (\R\times \Z). 
		 \label{eqG6.1}
	 \end{equation}
\end{corollary}
This is an immediate consequence of Lemmas \ref{lemG3}, \ref{lemG4} and \ref{lemG5}.


\section{Structure of spectrum, IV} \label{secH1}

In this section we continue to analyze the structure of the spectrum of a convex polytope
$\Omega$ in dimension $d=3$. Although we are mainly interested in the case when $\Omega$
is not a prism, we will not need to assume this in the present section. 
\subsection{}
Let $\Omega$ be a convex polytope in $\R^3$, centrally symmetric and with centrally
symmetric facets. Assume that $\Omega$ is in our ``standard position'', namely,
$\Omega=-\Omega$, $F$ is a facet of $\Omega$ contained in $\{x_1=\frac{1}{2}\}$, and
$F$ is symmetric about the point $(\frac{1}{2},0,0)$. Assume also that
$A=\{\frac{1}{2}\}\times \{\frac{1}{2}\}\times I$ is a subfacet of $F$, where
$I=[-\frac{1}{2},\frac{1}{2}]$. Hence $F=\{\frac{1}{2}\}\times\Sigma$, where
$\Sigma$ is a convex polygon in $\R^2$, $\Sigma=-\Sigma$, and
$\{\frac{1}{2}\}\times I$ is a facet of $\Sigma$.
\par
 Suppose that $\Lambda$ is a
spectrum for $\Omega$. Let $\Pi\subset \R^2$ be the set associated to the
spectrum $\Lambda$ defined in \secref{secD1}, and $\theta(s)$ be the function on
$\Pi$ given by \lemref{lemD2}. Recall that in \secref{secE1}, a new spectrum
$\Lambda'$ was constructed from the given spectrum $\Lambda$, by taking the weak
limit of a sequence of translates of $\Lambda$. The new spectrum $\Lam'$ was shown
(\lemref{lemE1}) to enjoy a particular structure, namely
\begin{equation}
	\label{eqH1.1}
	\Lambda' \subset \bigcup_{j\ge 0} (\Z+\theta_j)\times \Pi_j,
\end{equation}
where $\Pi_j$ are the components of the set $\Pi$,
and $\theta_j$  are respectively the values of the function $\theta(s)$
on these components. The sets $\Pi_j$ were shown (\corref{corD3}) to satisfy
\begin{equation}
	\label{eqH1.2}
	\Pi_k-\Pi_j \subset \big\{ \hat{\1}_\Sigma=0 \big\} \quad (j\ne k).
\end{equation}
\par
When we want to further analyze the structure of the spectrum in dimension
$d=3$, a new complication arises, that was not present in the case $d=2$. Namely,
the zero set $\big\{\hat{\1}_\Sigma=0\big\}$ is not known explicitly, except in the special
case when $\Sigma$ is the cube $I\times I$. In order to address this difficulty, a further
limiting procedure will now be performed on the spectrum $\Lambda'$, yielding a
third spectrum $\Lambda''$ of $\Omega$. 
\subsection{}
The new spectrum $\Lambda''$ is constructed as follows. Consider the sequence
of translates of the spectrum $\Lam'$ given by
\[ \Lambda'-r\cdot (0,1,0), \quad  r=1,2,3,\dots.\]
As in \secref{secE1} we may extract from this sequence a subsequence 
\begin{equation}
	\label{eqH1.3}
	\Lambda'-r_n\cdot (0,1,0), \quad r_n\to\infty,
\end{equation}
which converges weakly to some set $\Lambda''$, which is again a spectrum of $\Omega$.
\par
According to \eqref{eqH1.1} we may form a partition of the spectrum $\Lambda'$
into sets defined by
\begin{equation}
	\label{eqH1.3a}
	\Lambda'_j := \Lambda'\cap ( (\Z+\theta_j)\times \Pi_j).
\end{equation}
It would be convenient for us to know that for each $j$, the sequence of translates
\begin{equation}
	\label{eqH1.4}
	\Lambda_j'-r_n\cdot (0,1,0)
\end{equation}
of each component $\Lam_j'$ has a weak limit as $n\to\infty$. 
This does not follow automatically from the weak convergence of the sequence
\eqref{eqH1.3}, though, since we have not excluded the possibility that there may be
infinitely many $\theta_j$'s and that they may have accumulation points. Nevertheless, we
can assume that \eqref{eqH1.4} has a weak limit as $n\to\infty$ for each
$j$, simply by selecting a further subsequence if necessary.
\par
We shall denote by $\Lambda_j''$ the weak limit of \eqref{eqH1.4}. Observe that 
a point   $(t,u,v) \in \R^3$ belongs to $\Lambda_j''$ if and only if there
is a sequence $(t_n,u_n,v_n)\in\Lambda_j'$ such that 
\[ (t_n,u_n-r_n,v_n)\to (t,u,v), \quad n\to\infty. \]
Remark that while by \lemref{lemE4} none of the components $\Lam_j'$ may be empty, 
this is not true for the sets $\Lam_j''$ that we cannot exclude some of which to be empty.

\par
It follows from
\eqref{eqH1.3a} that 
\begin{equation}
	\label{eqH1.5}
\Lambda_j'' \subset \Lambda''\cap ( (\Z+\theta_j)\times \R^2), 
\end{equation}
hence the sets $\Lambda_j''$ are disjoint subsets of $\Lambda''$. Remark, however, that these sets do not
necessarily form a partition of $\Lambda''$, namely, their union need not be equal to the whole
$\Lambda''$. Again, this may happen only if there are infinitely many $\theta_j$'s. An example
of such a situation can be obtained if $\Omega$ is a prism whose base is a spectral set.
Indeed, we have seen in \exampref{expD4.2} that in such a case the
function $\theta(s)$ may attain countably many, arbitrary distinct values, and that the 
components $\Pi_j$ of the set $\Pi$ may be singletons. This implies that every $\Lambda_j''$ 
is empty, while $\Lambda''$ certainly cannot be empty being a spectrum for
$\Omega$.
\par
 This makes it necessary for us in general to consider also the subset of
$\Lambda''$ defined by 
\[ \Lambda''_{\infty}:= \Lambda''\setminus \bigcup_{j\ge 0}\Lambda_j''. \]

\begin{lem}
	\label{lemH1}
Let $(t,u,v) \in \R^3$. Then $(t,u,v)$
belongs to $\Lambda_\infty''$ if and only if there is a sequence $k_n\to\infty$, and 
for each $n$ there is a point $(t_n,u_n,v_n)\in\Lambda_{k_n}'$, such that 
\[ (t_n,u_n-r_n,v_n)\to (t,u,v), \quad n\to\infty. \]
\end{lem}

\begin{proof}
Suppose first that $(t,u,v)$ is a point in $\Lambda_\infty''$. Then
$(t,u,v) \in \Lam''$, and since $\Lambda''$ is the weak limit of 
\eqref{eqH1.3}, there exist $(t_n,u_n,v_n)\in\Lambda'$ such that
$(t_n,u_n-r_n,v_n)\to(t,u,v)$. Due to \eqref{eqH1.1} and \eqref{eqH1.3a}, for each
$n$ there is $k_n\ge 0$ such that 
$ (t_n,u_n,v_n)\in \Lambda_{k_n}'$. 
If $k_n\not\to \infty$, then $k_n$ admits infinitely often a certain value, say $k_n=j$
for infinitely many $n$'s. But this implies that $(t,u,v)$ must belong to the weak limit
of \eqref{eqH1.4}, and hence $(t,u,v)\in\Lambda_j''$, so it cannot
lie in $\Lambda_\infty''$. Hence we must have $k_n \to \infty$.
\par
Conversely, suppose that the point $(t,u,v)$ satisfies the condition in the lemma.
The condition implies that $(t,u,v)$ belongs to the weak limit of 
\eqref{eqH1.3}, hence $(t,u,v) \in \Lam''$. If  $(t,u,v)$ is not in
$\Lambda_\infty''$, then it belongs to one of the sets $\Lambda_j''$. 
But then we must have $k_n=j$ for all sufficiently large $n$, so $k_n\not\to \infty$,
a contradiction. Hence $(t,u,v) \in \Lambda_\infty''$.
\end{proof}

We also point out that the
inclusion \eqref{eqH1.5} is not necessarily an equality, as the right hand side of
\eqref{eqH1.5} may contain elements of $\Lambda_\infty''$. 
\subsection{} 
Now we establish some properties satisfied by the new spectrum $\Lambda''$ and its
components $\Lambda_k''$ ($0\le k\le\infty$). The first property is derived from the
condition \eqref{eqH1.2}.
\begin{lem}
	\label{lemH2}
	For each $0\le j$, $k\le \infty$, $j\ne k$, we have 
	\begin{equation}
		\label{eqH2.1}
		\Lambda_k''-\Lambda_j'' \subset \R\times \big\{ \hat{\1}_\Sigma =0
		\big\}.
	\end{equation}
\end{lem}
\begin{proof}
	By symmetry we may assume that $0\le j < k\le \infty$. Let $(t,u,v)\in
	\Lambda_j''$ and $(t',u',v')\in \Lambda_k''$. Then there exist two sequences
\[ (t_n,u_n,v_n)\in
	\Lambda_j', \quad  (t_n,u_n-r_n,v_n)\to (t,u,v), \]
and
\[ (t_n',u_n',v_n')\in
	\Lambda'_{k_n}, \quad  (t_n',u_n'-r_n,v_n')\to (t',u',v'), \]
where $k_n=k$ in the case when
	$k$ is finite, or $k_n\to\infty$ if $k=\infty$ (\lemref{lemH1}). In any case we have $k_n\ne j$ for
	all sufficiently large $n$. Since by \eqref{eqH1.3a} we have 
	\[ (u_n,v_n)\in \Pi_j,  \quad (u_n',v_n')\in\Pi_{k_n},\]
	then it follows from \eqref{eqH1.2} that 
	\[(t_n',u_n'-r_n,v_n')-(t_n,u_n-r_n,v_n)=
	(t_n'-t_n,u_n'-u_n,v_n'-v_n)\in \R\times \big\{\hat{\1}_\Sigma = 0\big\}.\] 
	Letting $n\to\infty$ we obtain 
\[ (t',u',v')-(t,u,v)\in \R\times
	\big\{\hat{\1}_\Sigma=0\big\}, \]
which confirms \eqref{eqH2.1}.
\end{proof}
\lemref{lemH2} shows that the structure 
\eqref{eqH1.2}  is basically preserved in the
 new spectrum $\Lambda''$ and its components $\Lambda_k''$ ($0\le k\le\infty$).
However, our motivation for introducing this new spectrum is due to the following lemma:
\begin{lem}
	\label{lemH3}
	Let $0\le j<\infty$, $0\le k\le\infty$, $k\ne j$. Then 
	\begin{equation}
		\label{eqH3.1}
		\Lambda_k'' - \R\times \Pi_j \subset (\R\times\Z\times\R)\cup
		(\R\times\R\times(\Z\setminus \{0\})).
	\end{equation}
	In other words, if $(u_0,v_0)\in \Pi_j$ and if $(t,u,v)\in \Lambda_k''$, then
	$u-u_0\in\Z$ or $v-v_0\in\Z\setminus \{0\}$. 
\end{lem}
This lemma is similar in spirit to \lemref{lemD2}. 
To see the resemblance between the two lemmas, recall that $\{\frac{1}{2}\}\times I$ is a
facet of the polygon $\Sigma$, and $\{\hat{\1}_I=0\}=\Z\setminus \{0\}$. 
The assertion of \eqref{eqH3.1} is equivalent to the statement that if $(u_0,v_0)\in\Pi_j$,
$(t,u,v)\in\Lambda_k''$, and if $\hat{\1}_I(v-v_0)\ne 0$, then $u\in\Z+u_0$. 
The proof is also similar to that of \lemref{lemD2}. 
\begin{proof}[Proof of \lemref{lemH3}]
	Let $(u_0,v_0)\in \Pi_j$ and $(t,u,v)\in\Lambda_k''$. Regardless of whether
	$k$ is finite or not, there is a sequence $k_n$ and there are points
	$(t_n,u_n,v_n)\in\Lambda_{k_n}'$ such that 
	\begin{equation}
		\label{eqH3.1b}
		(t_n,u_n-r_n,v_n)\to (t,u,v), \quad n\to\infty.
	\end{equation}
	Indeed, if $0\le k<\infty$ then $k_n=k$ for all $n$, while if $k=\infty$ then
	$k_n\to\infty$  (\lemref{lemH1}). In any case, we have $k_n\ne j$ for all sufficiently large
	$n$. Since $(t_n,u_n,v_n)\in \Lambda_{k_n}'$ we have $(u_n,v_n)\in\Pi_{k_n}$ by
	\eqref{eqH1.3a}. Hence by \eqref{eqH1.2} this implies that 
	\begin{equation}
		\label{eqH3.2}
		\hat{\1}_\Sigma(u_n-u_0,v_n-v_0)=0, 
	\end{equation}
	for all sufficiently large $n$.
\par
 Observe that since $r_n\to\infty$,
	\eqref{eqH3.1b} implies that also $u_n\to\infty$. Hence using
	\lemref{lemD1} for the polygon $\Sigma$ and its facet 
$\{\frac{1}{2}\}\times I$,
 it follows from \eqref{eqH3.2} that
	\[\sin \pi(u_n-u_0)\cdot \hat{\1}_I(v_n-v_0)\to 0, \quad n\to\infty.\]
	Indeed, the polygon  $\Sigma$ is  centrally symmetric and it has centrally
	symmetric facets, as the facets of $\Sigma$ are line segments, hence all the 
	conditions of \lemref{lemD1} are satisfied.
\par
Now suppose that $v-v_0\not\in \Z\setminus \{0\}$. Then $v-v_0$ is not
	contained in the zero set of $\hat{\1}_I$, and hence
	$|\hat{\1}_I(v_n-v_0)|$ remains bounded away from zero as $n\to\infty$.
	So we must have  $\sin \pi(u_n-u_0)\to 0$, 
	or equivalently,  $ \dist (u_n-u_0,\Z)\to 0$.
	But since $r_n$ is an integer,
	\eqref{eqH3.1b} implies that also
	$ \dist (u_n-u,\Z)\to 0$. It follows that
\[
	 \dist(u-u_0,\Z) \leq \dist(u_n-u_0,\Z) + \dist(u_n-u,\Z) \to 0.
\]
	We conclude that $u-u_0\in \Z$ as required.
\end{proof}
From the previous lemma it is easy to deduce the next one:
\begin{lem}
	\label{lemH4}
	For each $0\le j,k\le \infty$, $j\ne k$, we have 
	\begin{equation}
		\label{eqH4.1}
		\Lambda_k''-\Lambda_j''\subset (\R\times\Z\times\R)\cup
	(\R\times\R\times(\Z\setminus \{0\})). 
	\end{equation}
\end{lem}
Actually we will not use \lemref{lemH4} in what follows. We state
 it merely  to demonstrate an essential advantage  of the newly constructed spectrum
$\Lambda''$. On one hand, according to \eqref{eqH2.1} it basically inherits the structure
 of the previously constructed spectrum $\Lambda'$, while on the other hand, 
condition \eqref{eqH4.1} reveals an extra structure in $\Lam''$.
\par
Since the proof of \lemref{lemH4} is quite short, we include it for completeness.
\begin{proof}[Proof of \lemref{lemH4}]
	By symmetry we may assume that $0\le j<k\le \infty$. Let $(t,u,v)\in\Lambda_k''$.
	Then by \lemref{lemH3} the set $\R\times \Pi_j$ must be contained in 
	\begin{equation}
		\label{eqH4.2}
		(\R\times(u+\Z)\times\R)\cup
		(\R\times\R\times(v+(\Z\setminus \{0\}))).
	\end{equation}
	Due to \eqref{eqH1.3a} we have $\Lambda_j'\subset \R\times\Pi_j$, hence also the
	set $\Lambda_j'$
	is contained in \eqref{eqH4.2}. Since the set \eqref{eqH4.2} is
	invariant under translations by vectors in $\{0\}\times\Z\times\{0\}$, it follows
	that all the sets \eqref{eqH1.4} are also contained in \eqref{eqH4.2}, and hence
	the same is true for their weak limit $\Lambda_j''$. This implies that
	$\Lambda_j''-(t,u,v)$ is contained in the set on the right-hand side of
	\eqref{eqH4.1}. As $(t,u,v)$ was an arbitrary element of $\Lambda_k''$,
	this establishes \eqref{eqH4.1}.
\end{proof}


\section{Auxiliary lemmas} \label{secF3}

In this section we establish some specific facts about the spectrum of 
a convex polytope $\Omega$ that will be used later on. 
These facts are true in arbitrary dimension, so in the present
section we do not restrict the discussion to three dimensions.

\subsection{}
Let $\Omega\subset\R^d$ be a convex polytope. Let $F$ and $F'$ 
be two parallel facets of $\Omega$, and assume that
$ F\subset \{ x_1=\tfrac{1}{2} \}$, $ F'\subset \{
	x_1=-\tfrac{1}{2} \}$,
 and that $F$ is the image of $F'$ under translation by
	the vector $\vec{e}_1$. These assumptions imply that 
	\[ F=\left\{ \tfrac{1}{2} \right\}\times \Sigma, \quad F'=\left\{
	-\tfrac{1}{2} \right\}\times \Sigma, \]
	where $\Sigma$ is a convex  polytope in $\R^{d-1}$. 
\par
Assume also that $\Omega$ is
	spectral, and let $\Lambda$ be a spectrum for $\Omega$. 
\begin{lem} \label{lemF5}
	If $\Omega$ is not a prism, then $\Lambda$ cannot contain any set of the form 
	\begin{equation}
		\label{eqF5.1}
		(\Z+\theta)\times \{s\},
	\end{equation}
	where $\theta\in\R$ and $s\in\R^{d-1}$.
\end{lem}
\begin{proof}
	Suppose to the contrary that $\Lambda$ does contain a set of the form
	\eqref{eqF5.1}. This implies that the set $\Lambda-\Lambda$ contains $\Z\times
	\{0\}$. On the other hand, since $\Lambda$ is a spectrum for $\Omega$, 
	 the set $\Lambda-\Lambda$ must be contained in
	$\{\hat{\1}_\Omega=0\}\cup\{0\}$. We conclude that 
	\begin{equation}
		\label{eqF5.2}
		 \hat{\1}_\Omega (k,0) =0, \quad k\in\Z\setminus \{0\}.
	\end{equation}
\par
	For each $x\in \R$ denote by $\Omega_x$ the $(d-1)$-dimensional polytope obtained
	by the intersection of $\Omega$ with the hyperplane $\{x\}\times \R^{d-1}$, and
	let $\varphi(x)$ be the $(d-1)$-dimensional volume of $\Omega_x$. Then the
	function $\varphi$ vanishes off the interval
	$I=\left[-\frac{1}{2},\frac{1}{2}\right]$, it is continuous on $I$, and $\varphi(
	\frac{1}{2})=\varphi(-\frac{1}{2})=|\Sigma|$. Notice that, by convexity, $\Omega_x$
	contains $\{x\}\times \Sigma$ for every $x\in I$. In particular this implies that 
	$\varphi(x)\ge |\Sigma|$, $x\in I$.
\par
 It follows from the definition of the
	function $\varphi$ that its Fourier transform is given by 
	\[ \hat{\varphi}(t)=\hat{\1}_\Omega(t,0), \quad t\in\R. \]
	Combining this with \eqref{eqF5.2} we obtain that $\hat{\varphi}$ vanishes on
	$\Z\setminus\{0\}$. Since $\varphi$ is supported on $I$, this implies that
	$\varphi$ is orthogonal in $L^2(I)$ to all the exponentials $\{e_k\}$,
	$k\in\Z\setminus \{0\}$. But as the system $E(\Z)$ is orthogonal and complete in
	$L^2(I)$, this is possible only if $\varphi$ is constant  on $I$. Hence
	$\varphi(x)=|\Sigma|$ for all $x\in I$. In turn, this implies that
	$\Omega_x=\{x\}\times\Sigma$, $x\in I$. We conclude that $\Omega=I\times \Sigma$,
	and so $\Omega$ is a prism, a contradiction.
\end{proof}
\begin{remark*}
One can see from the proof that the only property of the set \eqref{eqF5.1} that was
actually used was that its difference set contains $\Z\times \{0\}$. Hence the lemma
remains true if \eqref{eqF5.1} is replaced by any other set for which the latter property
is satisfied.  
\end{remark*}

\subsection{}
Denote by $Q=I^{d-1}$ the unit cube in $\R^{d-1}$. As usual, $I$ is the interval
$\left[ -\frac{1}{2},\frac{1}{2} \right]$.
\begin{lem}
	\label{lemF6}
	Assume that $\Sigma$ contains $Q$. If $\Omega$ is not a prism, then $\Lambda$ cannot
	be covered by the union of two translates of $\Z^d$. 
\end{lem}
\begin{proof}
	Suppose to the contrary that $\Lambda$ is contained in the union of two translates
	of $\Z^d$. By translating $\Lambda$ we may assume that 
	\begin{equation}
		\label{eqF6.1}
		 	 \Lambda \subset \Z^d\cup  ( \Z^d + \tau  )
	\end{equation}
	for some $\tau \in \R^d$. 
	According to \lemref{lemF5}, the spectrum $\Lambda$ cannot contain the whole set
	$\Z\times \{0\}$. This implies that by further translating $\Lambda$ by a certain vector
	in $\Z\times \{0\}$, we may additionally assume that $\Lambda$ does not contain
	the origin. 
	\par
	Since $\Sigma$ is assumed to contain $Q$, and since by convexity
	$\Omega$ contains $I\times \Sigma$, then $\Omega$ must contain $I\times Q$, the
	unit cube in $\R^d$. Hence the function $f=\1_{I\times Q} $ is supported by
	$\Omega$. Consider the Fourier expansion \eqref{eqF5.3} of this function
	$f$. Since $\hat{f}$ vanishes on all the points of $\Z^d$ except the origin, and
	since the origin does not belong to $\Lambda$, it follows from \eqref{eqF6.1} that only exponentials
	$e_\lambda$ such that $\lambda\in\Lambda \cap  ( \Z^d+\tau  )$ may have a
	non-zero coefficient in the expansion \eqref{eqF5.3}. Hence by \lemref{lemPL6} the right-hand side of
	\eqref{eqF5.3} represents a function $\tilde{f}$ of the form 
	\[ \tilde{f}(x)=e^{2\pi i \dotprod{\tau}{x}} g(x), \quad x\in\R^d, \]
	where $g$ is some $\Z^d$-periodic function, and $f$ coincides with $\tilde{f}$ 
	a.e.\ on $\Omega$. Notice that $|g|=|\tilde{f}|=|f|=1$ a.e.\ on $I\times Q$. By the
	periodicity of $g$ this implies that $|g|=1$ a.e.\ on $\R^d$. Hence
	$|f|=|\tilde{f}|=|g|=1$ a.e.\ on $\Omega$. In particular, $f$ cannot vanish on any
	subset of $\Omega$ of positive measure. On the other hand, by the definition of
	$f$ it does vanish on $\Omega\setminus (I\times Q)$. This is possible only if
	$\Omega=I\times Q$, namely, $\Omega$ is the unit cube in $\R^d$. But this
	contradicts the  assumption that $\Omega$ is not a prism, so the proof is
	complete.
\end{proof}


\section{Structure of spectrum, V} \label{secH3}

In this section  we complete the analysis of the spectrum in dimension
$d=3$. 
\subsection{}
Our assumptions will be the following.
\par
 Let $\Omega\subset\R^3$ be a convex polytope,
centrally symmetric and with centrally symmetric facets. We assume that $\Omega$ is not a
prism. Suppose that $\Omega$ is in the ``standard position'', namely, $\Omega=-\Omega$,
$F$ is a facet of $\Omega$ contained in $\{x_1=\frac{1}{2}\}$, and $F$ is symmetric about
	the point $(\frac{1}{2},0,0)$. Hence $F=\{\frac{1}{2}\}\times \Sigma$, where
	$\Sigma$ is a convex polygon in $\R^2$ such that
	$\Sigma=-\Sigma$. We  assume that
	$A=\{\frac{1}{2}\}\times\{\frac{1}{2}\}\times I$ is a subfacet of $F$, where
	$I=[-\frac{1}{2},\frac{1}{2}]$, and therefore
	$\{\frac{1}{2}\}\times I$ is a facet of $\Sigma$. 
	We also suppose that $\interior(\Omega)$ intersects each one of the two open
	half-spaces $\{x_2<\frac{1}{2}\}$ and $\{x_2>\frac{1}{2}\}$. 
\par
	Suppose now that $\Lambda$ is a spectrum for $\Omega$. Let
	$\Pi$ be the set constructed from $\Lambda$ in \secref{secD1}, and
	$\theta(s)$ be the function on $\Pi$ given by \lemref{lemD2}. Let $\Lambda'$ be
	the spectrum for $\Omega$ constructed from $\Lambda$ in \secref{secE1}, and
	$\Lambda''$ be the spectrum constructed from $\Lambda'$ in \secref{secH1}. We
	shall continue to use the notations $\Pi_j$, $\theta_j$, $\Lambda'_j$, $\Lambda''_j$  and
	$\Lambda''_\infty$ with the same meaning as in the previous sections. 
\par
	Our goal in the present section is to prove that, under the assumptions above, the function $\theta(s)$ is
	necessarily constant on $\Pi$. 
	\subsection{}
	It will be convenient to introduce the following notation. Let
	\begin{equation}
		\label{eqHS.3}
		G:=(\Z\times \R)\cup (\R\times \Z), 
	\end{equation}
	and 
	\begin{equation}
		\label{eqHS.4}
		G_0:=(\Z\times \R)\cup  (\R\times(\Z\setminus \{0\}) ).
	\end{equation}
\begin{lem}
	\label{lemHS.5}
Let $\Pi_j$ $(0\le j< \infty)$ be one of the components of $\Pi$, 
 and let $0\le k\le\infty$, $k\ne j$. Then we have
	\begin{equation}
		\label{eqH5.1.2}
		\Lambda''_k \subset \R\times \bigcap_{s \in \Pi_j} (s + G_0).
	\end{equation}
Also, if the set $\Lam''_k$ is not empty, then we have
	\begin{equation}
		\label{eqH5.1.3}
		\Pi_j \subset \bigcap_{(t,s) \in \Lambda''_k} (s + G_0).
	\end{equation}
\end{lem}
In fact, each one of \eqref{eqH5.1.2} and \eqref{eqH5.1.3} is just a reformulation
of condition \eqref{eqH3.1}. Hence \lemref{lemHS.5} is a consequence
of \lemref{lemH3}.

\subsection{}
\begin{lem}
	\label{lemH11}
	If for some $0\le k\le \infty$, the set $\Lambda''_k$ is not empty, then 
	\begin{equation}
		\label{eqH11.1}
		\Lambda''_k-\Lambda''_k \not\subset \R\times G.
	\end{equation}
\end{lem}
\begin{proof}
	The proof is very similar to that of \corref{corG6}, and therefore it will only be
	outlined. The proof involves several steps. 
\par \emph{Step 1}. 
	Let $(t_0,s_0)$ be a point in $\Lambda''_k$, and let $f$ be the function defined
	by \eqref{eqE6.1}. Then the Fourier expansion 
	\begin{equation}
		\label{eqH11.2}
		f=\frac{1}{|\Omega|}\sum_{\lambda\in\Lambda''}\hat{f}(\lambda)e_\lambda
	\end{equation}
	of $f$ with respect to the spectrum $\Lambda''$ consists only of terms
	corresponding to $\lambda\in\Lambda''_k$. This follows from \lemref{lemH2} and the
	expression \eqref{eqE6.1.2} for the Fourier transform of $f$. 
\par \emph{Step 2}. 
	We have 
	\begin{equation}
		\label{eqH11.3}
		\Lambda''_k-\Lambda''_k \not\subset \R\times\Z\times\R.
	\end{equation}
	Indeed, if this is not true then by translating $\Lambda$ we may assume that
	$\Lambda''_k\subset \R\times \Z\times \R$. Hence from the Fourier expansion
	\eqref{eqH11.2} it follows (\lemref{lemPL6}) that $f$ coincides a.e.\ on $\Omega$ with a function
	$\tilde{f}$ on $\R^3$ which is periodic with respect to the vector $(0,1,0)$. 
	As in the proof of \lemref{lemG3} this leads to a contradiction to the assumption 
	that $\interior(\Omega)$ intersects both half-spaces $\{x_2<\frac{1}{2}\}$ and
	$\{x_2>\frac{1}{2}\}$. 
\par \emph{Step 3}. 
	We have 
	\begin{equation}
		\label{eqH11.4}
		\Lambda''_k-\Lambda''_k\not\subset \R\times\R\times\Z.
	\end{equation}
	In the same way, if this does not hold then by translating $\Lambda$ we can assume that
	$\Lambda''_k\subset \R\times\R\times \Z$. As in Step 2 this implies that
	$f$ coincides a.e.\ on $\Omega$ with a function $\tilde{f}$ on $\R^3$ which is
	periodic with respect to the vector $(0,0,1)$. As in the proof of
	\lemref{lemG4}, this together with \lemref{lemF4} implies that $\Omega$ must be a
	prism, a contradiction. 
\par \emph{Step 4}. 
	We have 
	\[ \Lambda''_k-\Lambda''_k\not\subset \R\times G. \]
	This follows by combining \eqref{eqH11.3}, \eqref{eqH11.4} and
	\lemref{lemG5}. 
\end{proof}

\subsection{}
\begin{lem}
	\label{lemH5}
	Let $s,s',s''$ be three points in $\R^2$, and 
	\begin{equation}
		\label{eqH5.1b}
		X=(s+G)\cap (s'+G)\cap (s''+G). 
	\end{equation}
	If the points $s,s',s''$ are distinct modulo $\Z^2$, then $X-X\subset G$. 
\end{lem}
This is not difficult to verify, and we omit the details.
\begin{lem}
	\label{lemH6}
	Suppose that there is a component $\Pi_j$ of the set $\Pi$ $(0\le j < \infty)$
	such that for any $0\le k\le \infty$, $k\ne j$, the set $\Lambda''_k$ is empty.
	Then $\Pi=\Pi_j$, namely $\Pi_j$ is the unique component of $\Pi$, and so the
	function $\theta(s)$ is constant on $\Pi$. 
\end{lem}
\begin{proof}
	The assumption means that $\Lambda''=\Lambda''_j$. By \eqref{eqH1.5} we therefore
	have 
	\[ \Lambda'' \subset (\Z+\theta_j)\times \R^2. \]
\par
	Consider the set of all points $s\in \R^2$ for which there is $t\in \Z+\theta_j$
	such that $(t,s)\in \Lambda''$. We claim that this set must contain at least three
	points which are distinct modulo $\Z^2$. Indeed, if this is not true then the
	spectrum $\Lambda''$ is contained in a union of two sets of the form
	\[ (\Z+\theta_j)\times (\Z^2+s), \quad s\in \R^2. \]
	But this would imply that $\Lambda''$ can be covered by the union of two
	translates of $\Z^3$, which is not possible according to \lemref{lemF6} since
	$\Omega$ is not a prism (notice that $\Sigma$ contains the cube $I\times I$, so
	we may use \lemref{lemF6}). Hence there must exist three points $(t,s)$, $(t',s')$,
	$(t'',s'')$ in the spectrum $\Lambda''$, such that $s,s',s''$ are distinct modulo
	$\Z^2$. 
\par
	Let $\Pi_k$,  $0\le k<\infty$, be any one of the components of $\Pi$ other than
	$\Pi_j$. Then by applying \eqref{eqH5.1.3} (with $j,k$ interchanged) we obtain
	\[ \Pi_k\subset (s+G)\cap (s'+G)\cap (s''+G). \]
	Using \lemref{lemH5} this implies that $\Pi_k-\Pi_k\subset G$, which is 
	impossible due to \corref{corG6}. It follows that $\Pi_j$ must be the unique
	component of $\Pi$. This means that $\theta(s)=\theta_j$ for all $s\in \Pi$,
	thus $\theta(s)$ is constant on $\Pi$. The lemma is therefore proved.
\end{proof}
\subsection{} 
At this point it will be useful to introduce the following:
\begin{definition}
	\label{defH8}
	Let $(s_0,s'_0)$ be a pair of points in $\R^2$ such that $s'_0-s_0\not\in G$. If
	$(s_1,s'_1)$ is another pair of points in $\R^2$, then  we say that $(s_1,s'_1)$ is
	\define{dual} to $(s_0,s'_0)$ if the following conditions are satisfied:
	\begin{enumerate-math}
		\item $s_1-s'_0$ and $s'_1-s_0$ are both in $\Z\times \R$;
		\item $s_1-s_0$ and $s'_1-s'_0$ are both in $\R\times \Z$. 
		\end{enumerate-math}
\end{definition}
For example, consider the pair $(s_0,s'_0)$ given by $s_0=(0,0)$, $s'_0=(\alpha,\beta)$,
where $\alpha,\beta$ are two real numbers which are both not in $\Z$. Then the pair
$(s_1,s'_1)$ given by $s_1=(\alpha,0)$, $s'_1=(0,\beta)$ is dual to $(s_0,s'_0)$.
\par
 It is not difficult to check that the duality relation just defined satisfies the following
properties:
\par
1.\ If $(s_1,s'_1)$ is dual to $(s_0,s'_0)$ then, since it was assumed that
$s'_0-s_0\not\in G$, it follows that also $s'_1-s_1\not\in G$.
\par
2.\ The duality relation is symmetric, that is, if $(s_1,s'_1)$ is dual to
$(s_0,s'_0)$, then also $(s_0,s'_0)$ is dual to $(s_1,s'_1)$. 
\par
3.\ Whether two given pairs are dual to each other or not depends only on the congruence classes of
the points modulo $\Z^2$. In other words, if $(s_1,s'_1)$ and $(s_2,s'_2)$ are two pairs such that
$s_2-s_1$ and $s'_2-s'_1$ are both in $\Z^2$, and if $(s_1,s'_1)$ is dual to a certain pair
$(s_0,s'_0)$, then also $(s_2,s'_2)$ is dual to $(s_0,s'_0)$. 
\par
4.\ For every pair $(s_0,s'_0)$ such that $s'_0-s_0 \not\in G$ there
exists a dual pair $(s_1,s'_1)$, and this dual pair is unique modulo $\Z^2$. 
\par
The reason for introducing the duality relation above is the following:
\begin{lem}
	\label{lemH9}
	Let $(s_0,s'_0)$ be a pair of points in $\R^2$ such that $s'_0-s_0\not\in G$. Then
	\begin{equation}
		\label{eqH9.1}
		(s_0+G)\cap (s'_0+G) = \Z^2+\{s_1,s'_1\},
	\end{equation}
	where $(s_1,s'_1)$ is any pair which is dual to $(s_0,s'_0)$. 
\end{lem}
This can be checked easily. It is also easy to see that \lemref{lemH9} implies:
\begin{lem}
	\label{lemH9.2}
Let $(s_0,s'_0)$ and $(s_1,s'_1)$ be two pairs of points in $\R^2$, such that 
$s'_0-s_0$  and $s'_1-s_1$ are both not in $G$. If the pairs
$(s_0,s'_0)$ and $(s_1,s'_1)$ are not dual to each other, then the set
\begin{equation}
	\label{eqH9.2.1}
	 Y = (s_0+G)\cap(s'_0+G)\cap (\Z^2+\{s_1,s'_1\})
\end{equation}
	is contained in a translate of $\Z^2$. 
\end{lem}

\subsection{}
\begin{lem}
	\label{lemH7}
	Suppose that the set $\Pi$ can be covered by the union of two translates of
	$\Z^2$. Then the function $\theta(s)$ is constant on $\Pi$.
\end{lem}
\begin{proof}
	By the assumption of the lemma there exist two points $s_0,s'_0\in\R^2$ such that 
	\begin{equation}
		\label{eqH7.1}
		\Pi\subset \Z^2+\{s_0,s'_0\}.
	\end{equation}
	Due to \eqref{eqH1.1} we have $\Lambda'\subset \R\times\Pi$, and together with
	\eqref{eqH7.1} this implies that $\Lambda'$ is contained in the set
	\begin{equation}
		\label{eqH7.2}
		\R\times  (\Z^2+\{s_0,s'_0\} ).
	\end{equation}
	Hence all the sets in \eqref{eqH1.3}, as well as their weak limit $\Lambda''$, are
	also contained in \eqref{eqH7.2}.
\par
 The set $\Pi$ has at least one component
	$\Pi_0$. Since by \corref{corG6} we have $\Pi_0-\Pi_0\not\subset G$, we may assume
	that $s_0,s'_0$ both belong to $\Pi_0$ and that $s'_0-s_0\not\in G$. Hence using
	\eqref{eqH5.1.2} for $j=0$ we conclude that 
	\[ \Lambda''_k\subset \R\times  ( (s_0+G)\cap (s'_0+G) )\]
	for every $1\le k\le\infty$. In turn, by \lemref{lemH9} this implies that
	$\Lambda''_k$ is contained in a set of the form 
	\begin{equation}
		\label{eqH7.3}
		\R\times  (\Z^2+\{s_1,s'_1\} ),
	\end{equation}
	where $(s_1,s'_1)$ is a pair which is dual to $(s_0,s'_0)$. 
\par
We conclude that for
	every $1\le k\le\infty$, the set $\Lambda''_k$ is contained in both
	\eqref{eqH7.2} and \eqref{eqH7.3}, hence $\Lambda''_k$ must be the empty set. Now
	\lemref{lemH6} allows us to deduce that $\Pi_0$ is the unique component of
	$\Pi$, and that $\theta(s)$ is a constant function on $\Pi$. The lemma is thus
	proved.
\end{proof}
\begin{lem}
	\label{lemH10}
	Suppose that one of the components $\Pi_j$ of $\Pi$ cannot be covered by the union
	of two translates of $\Z^2$. Then the function $\theta(s)$ is constant on
	$\Pi$. 
\end{lem}
\begin{proof}
	The assumption means that the component $\Pi_j$ contains three points $s,s',s''$
	which are distinct modulo $\Z^2$. Hence by \lemref{lemH5} the set $X$ defined by
	\eqref{eqH5.1b} satisfies $X-X\subset G$. By \eqref{eqH5.1.2}, for any $0\le k\le
	\infty$, $k\ne j$, we have $\Lambda''_k\subset \R\times X$, so it follows that 
	\[ \Lambda''_k-\Lambda_k''\subset \R\times G. \]
	But according to \lemref{lemH11} this is possible only if $\Lambda''_k$ is empty.
	We conclude that all the sets $\Lambda''_k$ such that $0\le k\le \infty$,
	$k\ne j$, are empty. By \lemref{lemH6} this implies that $\Pi_j$ is the unique
	component of $\Pi$, and $\theta(s)$ is constant on $\Pi$, as we had to show. 
\end{proof}

\subsection{}
\begin{lem}
	\label{lemH12}
	Suppose that the function $\theta(s)$ is not constant on $\Pi$. Then there exist
	two components $\Pi_{j_0}$ and $\Pi_{j_1}$ $(j_0 \neq j_1)$ of the set $\Pi$, and
	there are points $s_0,s'_0\in \Pi_{j_0}$ and $s_1,s'_1\in\Pi_{j_1}$, such that:
	\begin{enumerate-math}
	\item \label{lemH12.item1}
		$\Pi_{j_0}$ is contained in the set 
		\begin{equation}
			\label{eqH12.11}
			X_0 := \Z^2+\{s_0,s'_0\},
		\end{equation}
		while $\Pi_{j_1}$ is contained in 
		\begin{equation}
			\label{eqH12.12}
			X_1:= \Z^2+\{s_1,s'_1\};
		\end{equation}
	\item  \label{lemH12.item2}
		$\Lambda''_{j_0}\subset (\Z+\theta_{j_0})\times X_0$
		and $\Lambda''_{j_1}\subset 
		(\Z+\theta_{j_1})\times X_1$;
	\item  \label{lemH12.item3}
		the two pairs $(s_0,s'_0)$ and $(s_1,s'_1)$ are dual to each other;
	\item \label{lemH12.item4}
		$\Lambda''_k$ is empty for every $0\le k\le\infty$, $k\ne j_1$, $k\ne j_2$.
	\end{enumerate-math}
\end{lem}
\begin{proof}
	Assume that the function $\theta(s)$ is not constant on $\Pi$. Let $\Pi_{j_0}$ be
	one of the components of $\Pi$. By \corref{corG6} we have
	$\Pi_{j_0}-\Pi_{j_0}\not\subset G$, hence there exist two points $s_0,s'_0$ in
	$\Pi_{j_0}$ such that $s'_0-s_0\not\in G$. Observe that by \lemref{lemH10} the
	component $\Pi_{j_0}$ must be contained in the union of two translates of
	$\Z^2$, which are necessarily given by $\Z^2+s_0$ and $\Z^2+s'_0$. That is, 
	\begin{equation}
		\label{eqH12.2}
		\Pi_{j_0}\subset \Z^2+\{s_0,s'_0\}.
	\end{equation}
\par
	By \lemref{lemH7}, the set $\Pi$ cannot be covered by the union of two translates
	of $\Z^2$. Hence the set $\Pi$ must contain some point $s_1$ which is distinct
	modulo $\Z^2$ from both $s_0$ and $s'_0$. According to \eqref{eqH12.2}, the new
	point $s_1$ cannot belong to $\Pi_{j_0}$, hence it belongs to some other
	components $\Pi_{j_1}$. 
\par
Using \eqref{eqH5.1.2} it follows that for every $0\le k\le
	\infty$, $k\ne j_0$, $k\ne j_1$, we have 
	\[ \Lambda''_k\subset \R\times  ( (s_0+G)\cap (s'_0+G) \cap (s_1+G) ).\]
	But then \lemref{lemH5} implies that $\Lambda''_k-\Lambda''_k\subset \R\times G$.
	According to \lemref{lemH11} this is not possible unless $\Lambda''_k$ is empty.
	We conclude that all the sets $\Lambda''_k$, where $0\le k\le \infty$, $k\ne j_0$,
	$k\ne j_1$, are empty. 
\par
Due to \corref{corG6}, the component $\Pi_{j_1}$ cannot be
	contained in the set $\Z^2 + s_1$, hence there is another point $s'_1$ in
	$\Pi_{j_1}$ which is not congruent to $s_1$ modulo $\Z^2$. It then follows from
	\lemref{lemH10} that 
	\begin{equation}
		\label{eqH12.3}
		\Pi_{j_1}\subset \Z^2+\{s_1,s'_1\}.
	\end{equation}
	In turns, this implies that we must have $s'_1-s_1\not\in G$, again by \corref{corG6}.
\par
	Recalling the definition of the sets $\Lambda''_{j_0}$ and $\Lambda''_{j_1}$,
	the conditions \eqref{eqH12.2} and \eqref{eqH12.3} now imply that the property
	\ref{lemH12.item2} in the lemma is satisfied.
\par
 It remains to show that the pairs
	$(s_0,s'_0)$ and $(s_1,s'_1)$ are dual to each other. If this is not the case, then
	by \lemref{lemH9.2} the set $Y$ defined by \eqref{eqH9.2.1} 
	is contained in a translate of $\Z^2$. But we have $\Lambda''_{j_1}\subset
	(\Z+\theta_{j_1}) \times Y$, due to \eqref{eqH5.1.2} and  property
	\ref{lemH12.item2}. This implies that
	$\Lambda''_{j_1}-\Lambda''_{j_1}\subset \Z\times\Z^2$, and consequently
	$\Lambda''_{j_1}$ must be empty by \lemref{lemH11}. In a completely similar way we
	can also deduce that $\Lambda''_{j_0}$ must be empty. But this yields that all the
	sets $\Lambda''_k$, for every $0\le k\le\infty$, are empty, which is impossible
	since $\Lambda''$ cannot be empty being a spectrum for $\Omega$. This
	contradiction confirms that $(s_0,s'_0)$ and $(s_1,s'_1)$ must be dual to each
	other, and concludes the proof.
\end{proof}

\subsection{}
\begin{lem}
	\label{lemH13}
	The function $\theta(s)$ is necessarily constant on $\Pi$. 
\end{lem}
\begin{proof}
	Suppose to the contrary that this is not the case. Then by \lemref{lemH12} there are two
	components $\Pi_{j_0}$ and $\Pi_{j_1}$ $(j_0 \neq j_1)$  of the set $\Pi$, and there are points
	$s_0,s'_0\in\Pi_{j_0}$ and $s_1,s'_1\in\Pi_{j_1}$ satisfying all the properties
	\ref{lemH12.item1}--\ref{lemH12.item4} of that lemma.
\par
 By translating the spectrum
	$\Lambda$ by a vector in $\{0\}\times\R^2$ we may assume that $s_0=(0,0)$.
	Since $s'_0-s_0 \notin G$, we have
	$s'_0=(\alpha,\beta)$ for certain real numbers $\alpha,\beta$ none of which is an
	integer. Since the pair $(s_1,s'_1)$ is dual to $(s_0,s'_0)$, it follows that
	$s_1$ and $s'_1$ are congruent modulo $\Z^2$ to the points $(\alpha,0)$ and
	$(0,\beta)$ respectively. In other words, we have $s_1\in \Z^2+(\alpha,0)$ and
	$s'_1\in\Z^2+(0,\beta)$. 
\par
By further translating $\Lambda$ by a vector in
	$\R\times \{(0,0)\}$ we may also assume that $\theta_{j_0}=0$. It will be
	convenient to denote $\theta:=\theta_{j_1}$ (notice that we then have $0<\theta<1$,
	since $\theta_{j_0}$ and $\theta_{j_1}$ are different numbers).
\par
 According to
	\lemref{lemF5}, the spectrum $\Lambda''$ cannot contain the whole set $\Z\times
	\{(0,0)\}$. This implies that by translating $\Lambda$ once more by some vector in
	$\Z\times \{(0,0)\}$ we may additionally assume that $\Lambda''$ does not contain
	the origin $(0,0,0)$. 
\par
By property \ref{lemH12.item2} from \lemref{lemH12} we have
	\begin{equation}
		\label{eqH13.1}
		\Lambda''_{j_0} \subset \Z\times  ( \Z^2+\{(0,0),(\alpha,\beta)\} ).
	\end{equation}
	Hence each point in $\Lambda''_{j_0}$ belongs to one of two possible types:
	\begin{enumerate-text}
		\item \label{type1}
			Points of the form $(k,n,m)$ where $k,n,m$ are integers, not
			all of which are zero (that $k,n,m$ cannot all be zero follows from the assumption that $\Lambda''$ does not
			contain the origin);
		\item \label{type2}
			Points of the form $(k,n+\alpha,m+\beta)$ where $k,n,m$ are
			integers. 
	\end{enumerate-text}
\par
	By the same property \ref{lemH12.item2} from \lemref{lemH12}, we also have
	\begin{equation}
		\label{eqH13.2}
		\Lambda''_{j_1}\subset  (\Z+\theta )\times
		 (\Z^2+\{(\alpha,0),(0,\beta)\}  ).
	\end{equation}
\par
	Notice that so far, we have always used \eqref{eqH5.1.2} and \eqref{eqH5.1.3}
	with the set $G_0$ on the right hand side
	actually replaced by $G$ (which is valid since $G_0$ is a subset of
	$G$). However, at this point the fact that $G_0$, and not just $G$,
	appears on the right hand side of \eqref{eqH5.1.2} will be important. We apply
	\eqref{eqH5.1.2} with $j=j_0$ and $k=j_1$, and use the assumption that $(0,0)=s_0\in
	\Pi_{j_0}$, to conclude that 
	\begin{equation}
		\label{eqH13.3}
		\Lambda''_{j_1} \subset \R \times G_0.
	\end{equation}
\par
	It then follows from \eqref{eqH13.2} and \eqref{eqH13.3} that also each point in
	$\Lambda''_{j_1}$ belongs to one of two possible types:
	\begin{enumerate-text}
		\setcounter{enumi}{2}
		\item \label{type3}
			Points of the form $(k+\theta,n+\alpha,m)$ where $k,n,m$ are
			integers, and $m$ is non-zero (that $m$ cannot be zero follows
			from \eqref{eqH13.3} and the fact that $\alpha$ is not an
			integer);
		\item \label{type4}
			Points of the form $(k+\theta,n,m+\beta)$ where $k,n,m$ are
			integers.
	\end{enumerate-text}
\par
	By property \ref{lemH12.item4} of \lemref{lemH12}, the spectrum $\Lambda''$ is the
	union of the two disjoint sets $\Lambda''_{j_0}$ and $\Lambda''_{j_1}$. We
	conclude that each point of $\Lambda''$ belongs to one of the four types 1, 2, 3 and 4
	described above. 
\par
	Now
	consider the function \[ f(x,y,z):= \1_I(x)\1_I(y)\1_I(z), \quad (x,y,z)\in \R^3, \]
 where
	$I=[-\frac{1}{2},\frac{1}{2}]$, namely, $f$ is the indicator function of the unit
	cube in $\R^3$. Then $f$ is supported by $\Omega$. Consider the Fourier expansion 
	\begin{equation}
		\label{eqH13.4}
		f=\frac{1}{|\Omega|}\sum_{\lambda\in\Lambda''}\hat{f}(\lambda)e_\lambda
	\end{equation}
	of $f$ with respect to the spectrum $\Lambda''$. Since we have 
	\[ \hat{f}(t,u,v)=\hat{\1}_I(t) \, \hat{\1}_I(u) \, \hat{\1}_I(v), \quad (t,u,v)\in \R^3, \]
	it follows that $\hat{f}(t,u,v)=0$ whenever at least one of $t,u,v$ is a non-zero
	integer. This implies that $\hat{f}$ vanishes on all the points of $\Lambda''$
	which belong to types 1 and 3. Hence only exponentials
	$e_\lambda$ such that $\lambda$ is of type $2$ or $4$ may have a non-zero
	coefficient in the expansion \eqref{eqH13.4}.
\par
 It follows (\lemref{lemPL6}) that the right-hand side
	of \eqref{eqH13.4} is a function $\tilde{f}$ of the form 
	\begin{equation}
		\label{eqH13.5}
		\tilde{f}(x,y,z)=e^{2\pi i (\alpha y+\beta z)}g(x,y,z)+e^{2\pi i
		(\theta x+\beta z)}h(x,y,z), \quad (x,y,z)\in \R^3,
	\end{equation}
	where $g$ and $h$ are $\Z^3$-periodic functions, and $f$ coincides with
	$\tilde{f}$ a.e.\ on $\Omega$. Notice that it follows from \eqref{eqH13.5} that
	the function $|\tilde{f}|$ is periodic with respect to the vector $(0,0,1)$. Since
	we have $|\tilde{f}|=|f|=1$ a.e.\ on $I\times I\times I$, the periodicity of
	$|\tilde{f}|$ implies that $|\tilde{f}|=1$ a.e.\ on $I\times I\times \R$. Hence
	$|f|=|\tilde{f}|=1$ a.e.\ on the set $\Omega\cap (I\times I\times \R)$. On the
	other hand, by its definition $f$ vanishes on the set 
	\[ \Omega\cap \big( I\times I\times (\R\setminus I)\big), \]
	so the latter set must have measure zero. We conclude that 
	\begin{equation}
		\label{eqH13.6}
		\Omega\cap (I\times I\times \R)= I\times I\times I.
	\end{equation}
\par
	Since $\Omega$ contains the prism $I\times \Sigma$, and since
	$\{\frac{1}{2}\}\times I$ and $\{-\frac{1}{2}\}\times I$ are facets of $\Sigma$, it
	follows from \eqref{eqH13.6} that $\Sigma=I\times I$. Moreover, we obtain that the
	intersection of $\Omega$ and the slab $\R\times I\times \R$ coincides with
	$I\times \Sigma$. However, by \lemref{lemF4} this contradicts our assumption that
	$\Omega$ is not a prism. This completes the proof. 
\end{proof}


\section{Spectral convex polytopes in $\R^3$ tile by translations} \label{secG2}

Based on the results obtained in the previous sections, we can now deduce:

\begin{thm}\label{thmG1}
	Let $\Omega$ be a convex polytope in $\R^3$. If $\Omega$ is spectral, then it
	tiles by translations.	
\end{thm}
\subsection{}
By Theorems \ref{thmA1} and \ref{thmA2}, the polytope $\Omega$
must be centrally symmetric and have centrally symmetric facets. Since
\thmref{thmG1} was already proved in the case when $\Omega$ is a prism (\thmref{thmF1}),
it remains to consider the case when $\Omega$ is not a prism. 
\begin{lem} \label{lemG2}
	Let $\Omega$ be a convex polytope in $\R^3$, centrally symmetric and with
	centrally symmetric facets, which is not a prism. If $\Lambda$ is a spectrum of
	$\Omega$, then 
	\begin{equation}
		\dotprod{\Lambda-\Lambda}{\tau_F}\subset\Z 
		\label{eqG2.1}
	\end{equation}
	for every facet $F$ of $\Omega$. 
\end{lem}
This result is the three-dimensional analog of \lemref{lemE9}. By combining
\lemref{lemG2} with \corref{corC3} we immediately obtain that
$\Omega$ tiles by translations,  hence it only remains to prove the lemma. 
\subsection{}
\lemref{lemG2} is a direct consequence of our previous results:
\begin{proof}[Proof of \lemref{lemG2} ]
	Let $F$ be a facet of $\Omega$. We must show that if $\Lambda$ is a spectrum of
	$\Omega$, then it satisfies condition \eqref{eqG2.1}. Since $\Omega$ is not a
	prism, we may use \lemref{lemF3} to select a subfacet $A$ of $F$ such that
	$\interior(\Omega)$ intersects each one of the two open half-spaces bounded by the
	hyperplane $H_{F,A}$. 
\par
By applying an affine transformation we may suppose
	that $\Omega$ is in our ``standard position'', namely, $\Omega=-\Omega$,
	$F=\{\frac{1}{2}\}\times\Sigma$ where $\Sigma$ is a convex polygon in $\R^2$,
	$\Sigma=-\Sigma$, and $A=\{\frac{1}{2}\}\times\{\frac{1}{2}\}\times I$, where
		$I=\left[-\frac{1}{2},\frac{1}{2}\right]$. The hyperplane 
		$H_{F,A}$ is therefore given by $\{x_2=\frac{1}{2}\}$, and hence $\interior(\Omega)$ intersects
		both half-spaces $\{x_2<\frac{1}{2}\}$ and $\{x_2>\frac{1}{2}\}$. We also
		have $\tau_F=(1,0,0)$, so that condition \eqref{eqG2.1} becomes 
		\begin{equation}
			\label{eqG2.3}
			\Lambda-\Lambda \subset \Z\times \R^2.
		\end{equation}
\par
		Let $\Pi$ be the set constructed from $\Lambda$ in \secref{secD1}, and
		$\theta(s)$ be the function on $\Pi$ given by \lemref{lemD2}. Since all
		the assumptions of \secref{secH3} are satisfied, we may apply
		\lemref{lemH13}, which yields that the function $\theta(s)$ is constant on
		$\Pi$. By \corref{corE7} this implies that \eqref{eqG2.3} holds, which concludes
		the proof.
\end{proof}


\section{Uniqueness of the spectrum} \label{secG3}
The approach that was used above to prove that in dimensions $d=2,3$
any spectral convex polytope $\Omega$ can tile by translations, also allows us to establish that, 
except in the case when $\Omega$ is a prism, the spectrum is unique up to translation.
\subsection{}
To prove this we use the following lemma, which is valid in any dimension
$d$ (not just $d=2,3$). 
\begin{lem}
	\label{lemG7}
	Let $\Omega\subset \R^d$ be a convex polytope, centrally symmetric and with
	centrally symmetric facets. Suppose that $\Omega$ has a spectrum $\Lam$
	satisfying \eqref{eqC1.2}  for every 
	facet $F$ of $\Omega$. Then $\Lambda$ is a translate of the
	lattice $T^*$, the dual of the lattice $T$ given by \eqref{eqC1.1}. 
\end{lem}
\begin{proof}
	By \corref{corC3}, the set $T$ given by \eqref{eqC1.1} is a lattice, and $\Omega+T$ is a tiling.
	Hence by Fuglede's theorem the dual lattice $T^*$  is a spectrum for $\Omega$. 
	By translating $\Lambda$ we may assume that it contains the origin. So \eqref{eqC1.2} implies that 
	\[ \dotprod{\Lambda}{\tau}\subset \Z, \quad \tau\in T. \]
	This means that $\Lambda$ is a subset of $T^*$. But since
	no proper subset of a spectrum
	can also be  a spectrum, we must therefore have $\Lambda=T^*$. The lemma is thus proved.
\end{proof}
From this lemma we immediately obtain the following
sufficient condition for a spectral convex polytope 
to admit a unique spectrum up to translation:
\begin{corollary}
	\label{corG7.1}
	Let $\Omega\subset \R^d$ be a convex polytope, centrally symmetric and with
	centrally symmetric facets. Assume that $\Omega$ is spectral, and that condition
	\eqref{eqC1.2} is satisfied for every spectrum $\Lambda$ of $\Omega$ and every
	facet $F$ of $\Omega$. Then $\Omega$ has a unique spectrum up to translation.
	More specifically, every spectrum $\Lambda$ of $\Omega$ is a translate of the
	lattice $T^*$. 
\end{corollary}

\subsection{}
The criterion just proved can now be applied to the following situations: 
\begin{thm}\label{thmG8}
	Let $\Omega$ be a spectral convex polygon in $\R^2$ which is not a parallelogram.
	Then $\Omega$ admits a unique spectrum up to translation.
\end{thm}
\begin{thm}\label{thmG9}
	Let  $\Omega$ be a spectral convex polytope in $\R^3$ which is not a prism. Then
	$\Omega$ admits a unique spectrum up to translation.
\end{thm}
Indeed, by Theorems \ref{thmA1} and \ref{thmA2}, the polytope $\Omega$ must be centrally
symmetric and have centrally symmetric facets. Hence \thmref{thmG8} follows from \lemref{lemE9} and
\corref{corG7.1}, while \thmref{thmG9} is a consequence of \lemref{lemG2} and \corref{corG7.1}.
\par
Remark that the assumptions that  $\Omega$ is not a parallelogram in $\R^2$, and that it
is not a prism in $\R^3$, are necessary in these
results. Indeed, we have seen in \exampref{expD4.2} that if $\Omega$ is a prism, then it admits 
infinitely many non  translation-equivalent spectra.


\section{Remarks and open problems} \label{secJ1}
\subsection{}
It would be interesting to extend \thmref{thmI1.2} to dimensions $d\ge 4$. 
\begin{problem}
	\label{probJ1.1}
	Let $\Omega$ be a convex polytope in $\R^d$ $(d\ge 4)$. Prove that if $\Omega$ is 
	spectral, then it can tile the space by translations.
\end{problem}
We know (Theorems \ref{thmA1} and \ref{thmA2}) that such an $\Omega$ must be centrally symmetric and
have centrally symmetric facets. 
\par
Using our previous results, the assertion in Problem \ref{probJ1.1} can be verified
for the class of \emph{four-dimensional convex prisms} 
(the polytopes $\Omega \subset \R^4$ which can be expressed as the Minkowski sum of a three-dimensional
convex polytope and a line segment):
\begin{thm}
	\label{thmJ1.2}
	Let $\Omega$ be a convex prism in $\R^4$. If $\Omega$ is spectral, then it can tile
	by translations. 
\end{thm}
Indeed, this follows from a combination of Theorems \ref{thmF2} and \ref{thmG1}, in the same way as
we have deduced \thmref{thmF1} from Theorems \ref{thmE8} and \ref{thmF2}.

\subsection{} 
It is conceivable that Problem \ref{probJ1.1} could be solved in the general case by an
appropriate development of our approach. However, there are certain difficulties
which should be addressed in extending our proof to higher dimensions.

\par
One problem is to identify the
class of polytopes that would play the role of the parallelograms in two dimensions, and of
the prisms in three dimensions. The spectral polytopes in these classes do not have a unique
spectrum up to translation, and it was therefore necessary to exclude them in Lemmas
\ref{lemE9} and \ref{lemG2}, and, for $d=3$, to prove by a different method that they can tile by
translations (\thmref{thmF1}).
\par
 Another problem in higher dimensions might be to obtain an analog of \lemref{lemH3}. In that lemma we have
used the fact that in three dimensions, all the subfacets of $\Omega$ are line segments,
and hence in particular they are also centrally symmetric. However, 
a spectral convex polytope $\Omega$ in $\R^d$  $(d\ge 4)$ need not have centrally symmetric $k$-dimensional
faces, for any $2\le k\le d-2$ (see \secref{subsecA2.1}).
\par
The latter problem disappears, though, if we impose the extra assumption that the
convex polytope $\Omega$ is a \define{zonotope}. Thus we propose
the following restricted version of Problem \ref{probJ1.1}.
\begin{problem}
	\label{probJ1.3}
	Let $\Omega$ be a zonotope in $\R^d$ $(d\ge 4)$. Prove that if $\Omega$ is
	spectral, then it tiles by translations. 
\end{problem}
\subsection{} 
It would also be interesting to know whether the conclusion of \thmref{thmI1.2} is true
for any convex body $\Omega$ (not assumed a priori to be a polytope).
The paper \cite{IKT03}  contains a proof that, in two dimensions, a spectral convex body $\Omega$ 
must be a polygon. As far as we know, no such a result has been proved in
dimensions $d\ge 3$.
\begin{problem}
	\label{probJ1.4}
	Let $\Omega$ be a convex body in $\R^d$. Prove that if $\Omega$ is
	a spectral set, then it must be a polytope. 
\end{problem}
It is known \cite{IKT01} that $\Omega$ cannot have a smooth boundary. Using the results in
\cite{GriLev16a} it follows that the assertion is also true if $\Omega$ is a cylindric 
convex body whose base has a smooth boundary.


\end{document}